\documentclass[12pt]{amsart}
\usepackage{amsmath}
\usepackage{amssymb}
\usepackage{latexsym}

\newcommand{\Zint}{\mathbb {Z}}    
     
\newcommand{\Rea}{\mathbb {R}}      
\newcommand{\Cplx}{\mathbb {C}}     
\newcommand{\halmos}{\rule{5pt}{5pt}}

\numberwithin{equation}{section}

\newtheorem{proposition}{\bf Proposition}[section]
\newtheorem{theorem}[proposition]{\bf Theorem}

\newtheorem{conjecture}{\bf Conjecture}

\begin{document}

\title[Middle convolution with irregular singularities]
{Introduction to middle convolution for differential equations with irregular singularities}
\author{Kouichi Takemura}
\address{Department of Mathematical Sciences, Yokohama City University, 22-2 Seto, Kanazawa-ku, Yokohama 236-0027, Japan.}
\thanks{Current address: Department of Mathematics, Faculty of Science and Technology, Chuo University, 1-13-27 Kasuga, Bunkyo-ku Tokyo 112-8551, Japan.
E-mail: takemura@math.chuo-u.ac.jp}
\dedicatory{Dedicated to Professor Tetsuji Miwa on his sixtieth birthday}
\begin{abstract}
We introduce middle convolution for systems of linear differential equations with irregular singular points,
and we presend a tentative definition of the index of rigidity for them.
Under some assumption, we show a list of terminal patterns of irreducible systems of linear differential equations by iterated application of middle convolution when the index is positive or zero.
\end{abstract}


\maketitle

\section{Introduction} 

Middle convolution was originally introduced by N. Katz in his book "Rigid local systems" \cite{Katz}, and
several others studied and reformulated it.
Dettweiler and Reiter \cite{DR1,DR2} defined middle convolution for systems of Fuchsian differential equations written as 
\begin{align}
& \frac{dY}{dz} =\left( \sum _{i=1}^r \frac{A^{(i)}}{z-t_i} \right) Y.
\label{eq:Fuchs}
\end{align}
Note that Eq.(\ref{eq:Fuchs}) has singularities at $\{ t_1, \dots ,t_r , \infty \}$ and they are all regular.
On the theory of middle convolution, the index of rigidity plays important roles and it is preserved by two operation, addition and middle convolution.
Addition is related to multiplying a function to solutions of differential equations, and middle convolution is related to applying Euler's integral transformation to them.
Middle convolution may change the size of differential equations.
It was essentially established by Katz \cite{Katz} that every irreducible system of Fuchsian differential equations whose index of rigidity is two is reduced to the system of rank one.
The procedure to obtain rank one system is called Katz' algorithm.
Subsequently it is shown that the system has integral representations of solutions which are calculated by following Katz' algorithm.

In this paper we study middle convolution for systems of linear differential equations with irregular singularities, which are written as
\begin{equation}
\frac{dY}{dz} =\left( -\sum _{j=1}^{m_0} A_j^{(0)} z^{j-1} + \sum _{i=1}^r \sum _{j=0}^{m_i} \frac{A_j^{(i)}}{(z-t_i)^{j+1}} \right) Y.
\label{eq:DEirr0}
\end{equation}
If $m_i=0$ $(i\neq 0)$ (resp. $m_0=0$), then the point $z=t_i$ (resp. the point $z=\infty $) is regular singularity.
In particular if $m_0=\dots =m_r=0$, then the system is Fuchsian.
There are several important linear differential equations which are not Fuchsian, e.g., Kummer's confluent hypergeometric equation and Bessel's equation.
We want to extend the theory of middle convolution to include those equations.

A pioneering work on middle convolution with irregular singularities was carried out by Kawakami \cite{Kaw}.
He focused on the differential equations
\begin{equation}
(zI_n - T) \frac{d\Psi }{dz} =A \Psi ,
\label{eq:okubo}
\end{equation}
which is called Okubo normal form if $T$ is a diagonal matrix.
Eq.(\ref{eq:okubo}) is called generalized Okubo normal form, if $T$ may not be diagonalizable. 
Kawakami constructed a map from generalized Okubo normal forms to linear differential equations with irregular singularities with the condition $m_0=0$,
and he showed that the map is surjective, i.e. every irreducible equation written as Eq.(\ref{eq:DEirr0}) with the condition $m_0=0$ corresponds to an irreducible equation of generalized Okubo normal form.
He considered middle convolution by using generalized Okubo normal forms, because Euler's integral transformation is on well with (generalized) Okubo normal forms.
Yamakawa \cite{Yam} gave a geometric interpretation to Kawakami's map and investigated middle convolution with irregular singularities with the condition $m_0 \leq 1$ by Harnad duality.
 
In this paper, we construct middle convolution including the case $m_0 \neq 0$ directly, i.e. without generalized Okubo normal forms nor Harnad duality. 
Our construction is explicit as we will discuss in sections \ref{sec:conv} and \ref{sec:mc}.

We firstly define convolution matrices, which are compatible with Euler's integral transformation (see Theorem \ref{thm:DRintegrepr}).
The size of convolution matrices is $nM$, where $n$ is the size of given matrices $A_j^{(i)} $ in Eq.(\ref{eq:DEirr0}) and $M= r+\sum _{i=0}^r m_i$.
The module defined by convolution matrices may not be irreducible.
We consider a quotient of convolution matrices to obtain irreducible modules in section \ref{sec:mc}, which leads to the definition of middle convolution.
We can describe the quotient spaces explicitly, and it is an advantage of our construction.
We propose a tentative definition of the index of rigidity, which is expected to be preserved by middle convolution.

We study middle convolution further on the case that $m_i \leq 1$ and $A_1^{(i)} $ is semi-simple for all $i$ in section \ref{sec:mileq1}.
In particular, we show that the index of rigidity is preserved by middle convolution on this case.
By applying middle convolutions and additions appropriately, the size of differential equations may be possibly decreased.  
Under the assumption of this section, we show that if the index of rigidity is positive and the system of differential equations is irreducible,
then the size of differential equations can be decreased to one by iterated application of middle convolution and addition.
Moreover we show a list of terminal patterns obtained by iterated application of middle convolution and addition for the case that the index of rigidity is zero.

We give some comments for future reference in section \ref{sec:future}. 

\section{Convolution} \label{sec:conv}
 
Let ${\bold A}=(A^{(0)} _{m_0}, \dots , A^{(0)} _{1}, A^{(1)} _{m_1} , \dots ,A^{(r)} _{0} )$ be a tuple of matrices acting on the finite-dimensional vector space $V$ $(\dim V=n)$.
The tuple ${\bold A}$ is attached with the system of differential equations (\ref{eq:DEirr0}).
We denote by $\langle {\bold A} \rangle $ the algebra generalted by $A^{(i)} _{j} $ $(i=0,\dots ,r,$ $ j=\delta _{i,0},\dots m_i)$.
The $\langle {\bold A} \rangle $-module $V$ (or $\langle {\bold A} \rangle $) is called irreducible if there is no proper subspace $W (\subset V)$ such that $A^{(i)}_j W \subset W $ for any $i,j$.

Set 
\begin{align}
& M= r+\sum _{i=0}^r m_i, \quad V'= \bigoplus _{i=0} ^{r}  V^{(i)} = \Cplx ^{n M} , \\
& V^{(0)}= V ^{\oplus m_0}, \quad  V^{(i)}= V ^{\oplus (m_i +1)} , \; (i\geq 1). \nonumber 
\end{align}
We fix $\mu \in \Cplx$ and define convolution matrices $ \tilde{A} _j^{(i)} $ $(i=0,\dots ,r,$ $ j=\delta _{i,0},\dots m_i)$
acting on $V'$ by
\begin{align}
& \left( 
\begin{array}{c}
u^{(0)} _{m_0} \\
\vdots \\
u^{(0)} _{1}  \\
u^{(1)} _{m_1}  \\
\vdots \\
u^{(r)} _{0} 
\end{array}
\right)
= \tilde{A} _j^{(i)} 
\left( 
\begin{array}{c}
v^{(0)} _{m_0} \\
\vdots \\
v^{(0)} _{1}  \\
v^{(1)} _{m_1} \\
\vdots \\
 v^{(r)} _{0}
\end{array}
\right) ,
\end{align}
where $v^{(i')} _{j'} , u^{(i')} _{j'} \in V$ $(i'=0,\dots ,r,$ $ j'=\delta _{i',0},\dots m_{i'})$ and $u^{(i')} _{j'} $ are given by 
\begin{align}
u^{(i')} _{j'}  = \left\{ 
\begin{array}{cl} 
\mu v ^{(i')} _{j'  -j } & i'=i, \; j' > j ,\\
\displaystyle \sum _{i'' =0} ^{r} \sum _{j''=\delta _{i'',0}}^{m_{i''}} A^{(i'')} _{j''} v ^{(i'')} _{j''} &  i=0, \; i'=i, \; j'=j ,\\
\displaystyle \mu v ^{(i')} _{0 } + \sum _{i'' =0} ^{r} \sum _{j''=\delta _{i'',0}}^{m_{i''}} A^{(i'')} _{j''} v ^{(i'')} _{j''} &  i\neq 0, i'=i, j'=j, \\
0 &  \mbox{ otherwise}.
\end{array}
\right.
\end{align}
Namely 
\begin{align}
& \tilde{A} _j^{(0)} = \left( 
\begin{array}{c}
\left.
\begin{array}{cccccc}
 & & & \mu I_n & & \\
 & & & & \ddots  & \\
 & & & & & \mu I_n \\ 
\end{array} 
\right\} \scriptstyle{m_0-j}
\qquad \qquad \qquad \qquad  \qquad 
\\
\begin{array}{cccccccccc}
\! \!  A^{(0)}_{m_0} & \cdots & A^{(0)}_{2} & A^{(0)}_{1} & A_{m_1}^{(1)} & \cdots & A_{0}^{(1)} & A_{m_2}^{(2)} & \cdots  & A_{0}^{(r)} \\
& & & & & & & & &\\
& & & & & & & & &\\
& & & & & & & & & 
\end{array}
\end{array}
\! \! \right) , \\
& \begin{array}{l}
\tilde{A} _j^{(i)} = \\
(i\geq 1)
\end{array}
\left( 
\begin{array}{c}
\left. \qquad \qquad \qquad \qquad \qquad 
\begin{array}{c}
 \\
 \\
 \\
\end{array}
\right\} \scriptstyle{m_0 +(m_1 +1) + \dots +(m_{i-1}+1)}
\\
\left.
\begin{array}{ccccc}
 & & \mu I_n & & \\
 & & &  \ddots   & \\
 & & & & \mu I_n\\ 
\end{array} 
\right\} \scriptstyle{m_i-j}
\quad \qquad \;
\\
\begin{array}{ccccccccccc}
\! \!  A^{(0)}_{m_0} & \cdots  & A^{(i-1)}_{\delta _{i,0}} & A_{m_i}^{(i)} & \cdots & A_{1}^{(i)} & A_{0}^{(i)} +\mu I_n & A_{m_{i+1}}^{(i+1)} & \cdots  & A_{0}^{(r)} \\
& & & & & & & & &\\
& & & & & & & & &\\
& & & & & & & & & 
\end{array}
\end{array}
\! \! \right) ,
\label{eq:tAji}
\end{align}
where $I_n$ is the unit matrix of size $n$. 
We denote the tuple of convolution matrices by $\tilde{\bold A} = (\tilde{A}^{(0)} _{m_0}, \dots , \tilde{A}^{(0)} _{1}, $ $ \tilde{A}^{(1)} _{m_1} , \dots ,\tilde{A}^{(r)} _{0})$.
Note that the definition of convolution matrices is a straightforward generalization of Dettweiler and Reiter \cite{DR1}.
Then we can show that the convolution corresponds to Euler's integral transformation on linear differential equations, which is also analogous to Dettweiler and Reiter \cite{DR2}.
\begin{theorem} 
\label{thm:DRintegrepr}
Assume that $Y= \left( \begin{array}{c} y_{1} (z) \\ \vdots \\ y_{n} (z) \end{array} \right)$ is a solution of
\begin{align}
\frac{dY}{dz} =\left( -\sum _{j=1}^{m_0} A_j^{(0)} z^{j-1} + \sum _{i=1}^r \sum _{j=0}^{m_i} \frac{A_j^{(i)}}{(z-t_i)^{j+1}} \right) Y.
\label{eq:DEirr}
\end{align}
Let $\gamma $ be a cycle such that $\int _{\gamma } \frac{d}{dw} ( r(w) y_{j} (w) (z-w)^{\mu } ) dw =0 $ for any $j$ and any rational function $r(w)$. 
Then the function $U$ defined by
\begin{align}
& U = \left(
\begin{array}{l}
U^{(0)} _{m_0} (z) \\
\vdots  \\
U^{(0)} _{1} (z) \\
U^{(1)} _{m_1} (z) \\
\vdots  \\
U^{(r)} _{0} (z) 
\end{array}
\right) , \;
\begin{array}{l}
U^{(0)} _{j} (z)  = - 
\end{array}
\left(
\begin{array}{c}
\int _{\gamma } w^{j-1} y_{1} (w) (z-w)^{\mu } dw \\
\vdots  \\
\int _{\gamma } w^{j-1} y_{n} (w) (z-w)^{\mu } dw 
\end{array}
\right) ,  \label{eq:integrepU} \\
& 
\begin{array}{l}
U^{(i)} _{j} (z)  =  \\
{(i\neq 0)}
\end{array}
\left(
\begin{array}{c}
\int _{\gamma } (w-t_i)^{-j-1} y_{1} (w) (z-w)^{\mu } dw \\
\vdots  \\
\int _{\gamma } (w-t_i)^{-j-1} y_{n} (w) (z-w)^{\mu } dw 
\end{array}
\right)  
\end{align}
satisfies 
\begin{equation}
\frac{dU}{dz} =\left( -\sum _{j=1}^{m_0} \tilde{A}_j^{(0)} z^{j-1} + \sum _{i=1}^r \sum _{j=0}^{m_i} \frac{\tilde{A}_j^{(i)}}{(z-t_i)^{j+1}} \right) U.
\label{eq:dYdztAzY0}
\end{equation}
\end{theorem}
The proof will be given in our forthcoming paper \cite{TakMI}.

Let $\gamma _{t}$ is a cycle turning the point $w=t$ anti-clockwise.
Then the contour $[\gamma _{z} ,\gamma _{t_i}] = \gamma _{z} \gamma _{t_i} \gamma _{z} ^{-1} \gamma _{t_i} ^{-1}$ for $i=0, \dots , r$ $(t_0=\infty )$ satisfies the condition of Theorem \ref{thm:DRintegrepr}.
We should also consider cycles which reflect the Stokes phenomena if there exists an irregular singularity. 

\section{Middle convolution and the index of rigidity} \label{sec:mc}

\subsection{Middle convolution}
We have defined convolution in section \ref{sec:conv}, although convolution may not preserve irreducibility.
We now consider a quotient of convolution.
We define the subspaces $\mathcal K$, $\mathcal L '(\mu )$ and $\mathcal L (\mu )$ of $V' = V^{\oplus M}$ by
\begin{align}
& 
\begin{array}{l}
{\mathcal K}^{(i)} =\\ 
(i \geq 1)
\end{array}
\left\{ \left( \! \! 
\begin{array}{c}
v ^{(i)} _{m_i} \\
v ^{(i)} _{m_i -1} \\
\vdots \\
v ^{(i)} _{0} 
\end{array} \! \! 
\right) \in V^{(i)} \: \vline \:
\left( 
\begin{array}{cccc}
A ^{(i)} _{m_i} & A ^{(i)} _{m_i -1}& \dots & A ^{(i)} _{0}  \\
0 & A ^{(i)} _{m_i }& \dots & A ^{(i)} _{1}  \\
0 & 0 & \ddots & \vdots \\
0 & \cdots  & 0 & A ^{(i)} _{m_i } 
\end{array}
\right) 
\left( \! \! 
\begin{array}{c}
v ^{(i)} _{m_i} \\
v ^{(i)} _{m_i -1} \\
\vdots \\
v ^{(i)} _{0} 
\end{array} \! \! 
\right) 
 =\left( 
\begin{array}{c}
0 \\
0 \\
\vdots \\
0
\end{array}
\right) 
\right\} , \\
& {\mathcal K}^{(0)} = \{0 \} (\subset V^{(0)}), \quad   {\mathcal K}= \bigoplus _{i=0} ^{r} {\mathcal K}^{(i)} , \nonumber \\
& {\mathcal L} ' (\mu ) = 
\left\{ \left( 
\begin{array}{c}
v ^{(0)} _{m_0} \\
\vdots \\
v ^{(0)} _{1} \\
v ^{(1)} _{m_1} \\
\vdots \\
v ^{(r)} _{0} 
\end{array}
\right) \: \vline \: 
\begin{array}{l}
\quad v^{(i)} _{j} =0 ,\quad (i \neq 0, j \neq 0), \\
\quad  v^{(1)} _{0}  =\dots  = v^{(r)} _{0} =-\ell , \\
 \left( 
\begin{array}{cccc}
A ^{(0)} _{m_0} &  \dots & A ^{(0)} _{1} & A ^{(0)} _{0} -\mu I_n \\
0 & A ^{(0)} _{m_0 }&  \dots & A ^{(0)} _{1}  \\
0 & 0 & \ddots & \vdots \\
0 & \cdots  & 0 & A ^{(0)} _{m_0 } 
\end{array}
\right) 
\left( 
\begin{array}{c}
v ^{(0)} _{m_0} \\
\vdots \\
v ^{(0)} _{1} \\
\ell
\end{array}
\right) 
 =\left( 
\begin{array}{c}
0 \\
\vdots \\
0 \\
0
\end{array}
\right) 
\end{array}
\right\} , \nonumber \\
& 
\begin{array}{l}
\\
{\mathcal L} (\mu ) = {\mathcal L} ' (\mu ) , \\
 (\mu \neq 0) 
\end{array} \quad 
 {\mathcal L} (0) = \left\{ \left( 
\begin{array}{c}
v ^{(0)} _{m_0} \\
\vdots \\
v ^{(0)} _{1} \\
v ^{(1)} _{m_1} \\
\vdots \\
v ^{(r)} _{0} 
\end{array}
\right) \: \vline \: \sum _{i =0} ^{r} \sum _{j=\delta _{i,0}}^{m_{i}} A^{(i)} _{j} v ^{(i)} _{j} =0 \right\} ,
\nonumber 
\end{align}
where $A_0^{(0)} =-(A^{(1)}_0 + \dots + A^{(r)}_0)$.
\begin{proposition}
We have $\tilde{A}^{(i)}_j {\mathcal K} \subset {\mathcal K} $, $\tilde{A}^{(i)}_j {\mathcal L} (\mu ) \subset {\mathcal L} (\mu ) $ and $\tilde{A}^{(i)}_j {\mathcal L} '(\mu ) \subset {\mathcal L} '(\mu ) $ for all $i,j$.
\end{proposition}
\begin{proof}
We show that $\tilde{A}^{(i)}_j {\mathcal L} '(\mu ) \subset {\mathcal L} '(\mu ) $.
Assume that  $v=(v^{(0)} _{m_0} \: \dots \: v^{(0)} _{1}\: v^{(1)} _{m_1} \: $ $ \dots \: v^{(r)} _{0} )^{\bold T} \in {\mathcal L}' (\mu )$.
Then  $v^{(i)} _{j} =0 $ $(i \neq 0, j \neq 0)$, $v^{(1)} _{0}  =\dots = v^{(r)} _{0} =-\ell $ and $( v^{(0)} _{m_i} \: \dots \: v^{(0)} _{1} \: \ell)^{\bold T}$ satisfies the definition of ${\mathcal L} '(\mu ) $.
In particular we have
\begin{align}
& \sum _{i '=0} ^{r} \sum _{j'=\delta _{i',0}}^{m_{i'}} A^{(i')} _{j'} v ^{(i')} _{j'} = \sum _{j'=1}^{m_0} A^{(0)} _{j'} v ^{(0)} _{j'} - \sum _{i =1} ^{r} A^{(i)} _{0} \ell \\
& = \sum _{j'=1}^{m_0} A^{(0)} _{j'} v ^{(0)} _{j'} + A^{(0)} _{0} \ell = \mu \ell . \nonumber
\end{align}
Then $\tilde{A}^{(0)}_j v $ $(j\neq 0)$ is written as
\begin{align}
\tilde{A}^{(0)}_j v = \left( 
\begin{array}{c}
\mu v ^{(0)} _{m_0-j} \\
\vdots \\
\mu v ^{(0)} _{1} \\
\sum _{j=1}^{m_0} A^{(0)} _{j} v ^{(0)} _{j} + \sum _{i =1} ^{r} A^{(i)} _{0} v ^{(i)} _{0} \\
0 \\
\vdots 
\end{array}
\right) 
 =\mu \left( 
\begin{array}{c}
v ^{(0)} _{m_0-j} \\
\vdots \\
 v ^{(0)} _{1} \\
\ell \\
0 \\
\vdots 
\end{array}
\right) ,
\end{align}
and 
$( v^{(0)} _{m_0-j} \: \dots \: v^{(0)} _{1} \:  \: \ell \; 0 \dots )^{\bold T}$ satisfies the definition of ${\mathcal L} '(\mu ) $.
Hence $\tilde{A}^{(0)}_j v \in {\mathcal L} '(\mu )$ for $j=1, \dots ,m_0$.
We also have $\tilde{A}^{(i)}_j v =0 $ for $i \neq 0$, because $ v^{(i)} _{m_i} =\dots =v^{(i)} _{1}=0$ and $\mu v^{(i)} _{0} + \sum _{i '=0} ^{r} \sum _{j'=\delta _{i',0}}^{m_{i'}} A^{(i')} _{j'} v ^{(i')} _{j'}  = -\mu \ell + \mu \ell =0$.  

The other cases are shown similarly.
\end{proof}
We define $mc_{\mu } (V)$ to be the $\langle mc_{\mu } (\bold A ) \rangle $-module $V^{\oplus M} /({\mathcal K} +{\mathcal L} (\mu ))$ where $mc_{\mu } (\bold A )$ is the tuple of matrices on $V^{\oplus M} /({\mathcal K} +{\mathcal L} (\mu ))$ whose actions are determined by $\tilde{\bold A}$, and we call it the middle convolution of $V$ with the parameter $\mu $.
The following propositions are analogues to Dettweiler and Reiter \cite{DR1}, which will be shown in our forthcoming paper \cite{TakMI}.
\begin{proposition} \label{prop:relKL} 
(i) If $\mu \neq 0$, then ${\mathcal K} \cap {\mathcal L} (\mu ) =\{0 \}$.\\
(ii) If $\mu =0$, then ${\mathcal K} + {\mathcal L} '(0)  \subset {\mathcal L} (0)$.\\
(iii) If the $\langle \bold A \rangle $-module $V$ is irreducible and $\mu =0$, then ${\mathcal K} \cap {\mathcal L} ' (0) =\{0 \}$ and $\dim {\mathcal K} + \dim {\mathcal L} ' (0)  \leq n ( M-1)$.
\end{proposition}
\begin{proposition}
Assume that the $\langle \bold A \rangle$-module $V$ is irreducible.\\
(i) $mc _{0} (V) \simeq V$ as $\langle \bold A \rangle$-modules.\\
(ii) The $\langle mc_{\mu } (\bold A ) \rangle $-module $mc _{\mu } (V)$ is irreducible and  $V \simeq mc _{- \mu } (mc _{\mu } (V))$ for any $\mu $.
\end{proposition}

\subsection{Addition}

Let $Y$ be a solution of Eq.(\ref{eq:DEirr0}).
Then the function 
\begin{align}
 Y' = \exp  \left( -\sum _{j=1}^{m_0} \frac{\mu ^{(0)} _{j}}{j} z^{j}  -\sum _{i=1}^r \sum _{j=1}^{m_i} \frac{\mu _j^{(i)}}{j(z-t_i)^{j+1}} \right) \prod _{i=1}^r (z-t_i)^{\mu _0^{(i)}} Y
\end{align}
satisfies the equation
\begin{align}
\frac{dY'}{dz} =\left( -\sum _{j=1}^{m_0} (A_j^{(0)} + \mu ^{(0)} _{j} I_n ) z^{j-1} + \sum _{i=1}^r \sum _{j=0}^{m_i} \frac{A_j^{(i)} +  \mu ^{(i)} _{j} I_n}{(z-t_i)^{j+1}} \right) Y'.
\label{eq:DEirrYp}
\end{align}
We now define addition for the tuple ${\bold A}=( A^{(0)} _{m_0}, \dots ,  A^{(0)} _{1}$, $A^{(1)} _{m_1} , \dots ,A^{(r)} _{0} )$ by
\begin{align}
& M_{\overline{\mu } } ({\bold A} ) = {\bold A} + \overline{\mu } I_n = (A^{(0)} _{m_0}+ \mu ^{(0)} _{m_0} I_n , \dots  , A^{(r)} _{0}  \! \! \! + \mu ^{(r)} _{0}I_n  ),
\end{align}
where ${\overline{\mu }} = (\mu ^{(0)} _{m_0}, \dots , \mu ^{(0)} _{1}, \mu ^{(1)} _{m_1} , \dots ,\mu ^{(r)} _{0} ) \in \Cplx ^{M}$.

\subsection{Index of rigidity}
Let 
${\bold A}=(A^{(0)} _{m_0}, \dots , A^{(0)} _{1}, A^{(1)} _{m_1} , \dots ,A^{(r)} _{0} )$ be a tuple of matrices acting on $V$.
Set
\begin{align}
A ^{(i)} =\left( 
\begin{array}{cccc}
A ^{(i)} _{m_i} & A ^{(i)} _{m_i -1}& \dots & A ^{(i)} _{0}  \\
0 & A ^{(i)} _{m_i }& \dots & A ^{(i)} _{1}  \\
0 & 0 & \ddots & \vdots \\
0 & \cdots  & 0 & A ^{(i)} _{m_i } 
\end{array}
\right) {\in {\rm End} ({V ^{\oplus (m_i +1 )})}\atop(i=0,\dots ,r)},
\end{align}
and 
\begin{align}
& 
{\mathcal C}^{(i)} =
\left\{
C^{(i)}  = 
\left.
\left( 
\begin{array}{cccc}
C ^{(i)} _{m_i} & C ^{(i)} _{m_i -1}& \dots & C ^{(i)} _{0}  \\
0 & C ^{(i)} _{m_i }& \dots & C ^{(i)} _{1}  \\
0 & 0 & \ddots & \vdots \\
0 & \cdots  & 0 & C ^{(i)} _{m_i } 
\end{array} 
\right)
\right|
A ^{(i)} C ^{(i)} = C^{(i)} A^{(i)}
\right\} .
\end{align}
We define the index of rigidity by 
\begin{align}
{\rm idx}({\bold A}) = \sum _{i=0}^r \dim ({\mathcal C}^{(i)} ) -\left( M -1 \right) (\dim (V))^2 ,
\end{align}
where $M= r+ \sum _{i=0} ^r  m_i$.
The condition $A ^{(i)} C ^{(i)} = C^{(i)} A^{(i)} $ is equivalent to
\begin{align}
& \sum _{j=0}^k \left( A^{(i)} _{m_i-j} C^{(i)} _{m_i-k+j} -C^{(i)} _{m_i-k+j} A^{(i)} _{m_i-j} \right) =0, \quad k=0,\dots ,m_i.
\label{eq:relrig}
\end{align}
The following proposition is readily obtained by Eq.(\ref{eq:relrig}):
\begin{proposition}
The index of rigidity is preserved by addition, i.e. 
${\rm idx}(M_{\overline{\mu } } ({\bold A} )) = {\rm idx}({\bold A}) $.
\end{proposition}
\begin{conjecture} \label{con:indrig}
If the $\langle \bold A \rangle$-module $V$ is irreducible, then the index of rigidity is preserved by middle convolution, i.e. 
${\rm idx}(mc _{\mu }({\bold A })) = {\rm idx}({\bold A})$.
\end{conjecture}
We will prove the conjecture for a special case in section \ref{sec:mileq1} (see Proposition \ref{prop:indrig}).

We define the local index of rigidity by
\begin{align}
{\rm idx}_i({\bold A}) = {\rm idx}_i [A^{(i)}_{m_i}, \dots , A^{(i)}_{0}] = \dim ({\mathcal C}^{(i)} ) -\left( m_i +1 \right) (\dim (V))^2 .
\end{align}
Then we have
\begin{align}
{\rm idx}({\bold A}) =  \sum _{i=0}^r {\rm idx}_i({\bold A}) + 2 (\dim (V))^2 .
\label{eq:idxlidx}
\end{align}
\subsection{Example} \label{subsec:ex}

We consider irreducible systems of differential equations of size two written as
\begin{align}
& \frac{dY}{dz} 
= \left( - A^{(0)}_1+\frac{A^{(1)}_0}{z} \right) Y , \; \;
Y=
\left(
\begin{array}{l}
y_{1}(z) \\
y_{2}(z) 
\end{array}
 \right) . 
\label{eq:dYdzA01zY0}
\end{align}
It has an irregular singularity at $z=\infty $ and a regular singularity at $z=0$. 

First we consider the case that $A^{(0)}_1$ is semi-simple.
It follows from irreducibility that $A^{(0)}_1 $ is not scalar.
By applying addition (i.e. multiplying $e^{\nu ' z} z^{\alpha '} $ to the solution $Y$) and gauge transformation (i.e. multiplying a constant matrix to the solution $Y$), we may assume that $A^{(0)}_1$ is a diagonal matrix with the eigenvalues $0$ and $-\nu (\neq 0 )$ and $A^{(1)}_0$ has the eigenvalues $0$ and $-\gamma  $.
Then we may set
\begin{align}
& A^{(0)}_1= \left(
\begin{array}{cc}
0 & 0 \\
0 & -\nu
\end{array}
\right) , \quad 
A^{(1)}_0=  \left(
\begin{array}{ll}
- \alpha  & k \\
\frac{\alpha(\gamma - \alpha )}{k}  & \alpha -\gamma  
\end{array}
\right) .
\label{eq:A0A1} 
\end{align}
It follows from irreducibility that $k \neq 0$ and $\alpha \neq 0$.
By eliminating $y_2 (z)$ in Eq.(\ref{eq:dYdzA01zY0}), we have a second order linear differential equation,
\begin{align}
& z\frac{d^2y_1}{dz^2} + (\gamma -\nu z)  \frac{dy_1}{dz} -\alpha \nu y_1=0.  \label{eq:y1DE} 
\end{align}
If $\nu = 1$, then Eq.(\ref{eq:y1DE}) represents the confluent hypergeometric differential equation.
Eq.(\ref{eq:y1DE}) for the case $\nu \neq 0$ reduces to the confluent hypergeometric differential equation by changing the variable $z'= \nu z$.
The index of rigidity for Eq.(\ref{eq:dYdzA01zY0}) is two, because $\dim ({\mathcal C}^{(0)} ) =4$ and $\dim ({\mathcal C}^{(1)}) =2$ which follows from $\nu \neq 0$ and 
$k \neq 0$.

We investigate middle convolution $mc_{\mu } $ for the matrices in Eq.(\ref{eq:A0A1}).
Convolution matrices are given as
\begin{align}
& \tilde{A}^{(0)}_1=
\left(
\begin{array}{cc}
A^{(0)}_1 & A^{(1)}_0 \\
0 & 0 
\end{array}
\right) =
\left(
\begin{array}{cccc}
0 & 0 & - \alpha  & k \\
0 & -\nu & \frac{\alpha(\gamma - \alpha )}{k}  & \alpha -\gamma \\
0 & 0 & 0 & 0 \\
0 & 0 & 0 & 0 
\end{array}
\right) ,\\
& 
\tilde{A}^{(1)}_0=
\left(
\begin{array}{cc}
0 & 0 \\
A^{(0)}_1 & A^{(1)}_0 +\mu I_2
\end{array}
\right) =
\left(
\begin{array}{cccc}
0 & 0 & 0 & 0 \\
0 & 0 & 0 & 0 \\
0 & 0 & - \alpha +\mu  & k \\
0 & -\nu & \frac{\alpha(\gamma - \alpha )}{k}  & \alpha -\gamma +\mu 
\end{array}
\right) . \nonumber
\end{align}
The dimension of the space ${\mathcal K} (\simeq {\mathcal K}_1)$ is one and the space is described as
\begin{align}
& {\mathcal K}= \left( 
\begin{array}{c}
0 \\
\mbox{Ker}(A^{(1)}_0) 
\end{array}
 \right) 
= \Cplx \left( 
\begin{array}{c}
0 \\
0 \\
k \\
\alpha 
\end{array}
 \right) .
\end{align}
The space ${\mathcal L} (\mu )$ $(\mu \neq 0) $ is described as
\begin{align}
{\mathcal L} (\mu ) & = 
\left\{ \left( 
\begin{array}{c}
v ^{(0)} _{1} \\
-\ell
\end{array}
\right) \: \vline \: 
\left( 
\begin{array}{cc}
A ^{(0)} _{1} & -A ^{(1)} _{0} -\mu  \\
0 & A ^{(0)} _{1} 
\end{array}
\right) 
\left( 
\begin{array}{c}
v ^{(0)} _{1} \\
\ell
\end{array}
\right) 
 =\left( 
\begin{array}{c}
0 \\
0
\end{array}
\right) 
\right\} \\
&  = \mbox{Ker} \left(
\begin{array}{cccc}
0 & 0 & - \alpha + \mu  & k \\
0 & -\nu  & \frac{\alpha(\gamma - \alpha )}{k}  & \alpha -\gamma +\mu  \\
0 & 0 & 0 & 0 \\
0 & 0 & 0 & -\nu 
\end{array}
\right) .
\nonumber
\end{align}
Hence the dimension of the space ${\mathcal L} (\alpha )$ for the case $\mu \neq \alpha $ (resp. $\mu = \alpha  $) is one (resp. two).
We concentrate on the case $\mu =\alpha (\neq 0) $.
A basis of the space ${\mathcal L} (\alpha ) $ is given by
\begin{align}
\left( 
\begin{array}{c}
1 \\
0 \\
0 \\
0 
\end{array}
  \right) , \; 
\left( 
\begin{array}{c}
0 \\
\alpha (\gamma - \alpha ) \\
k\nu \\
0
\end{array}
\right) .
\end{align}
Set
\begin{align}
& S= \left(
\begin{array}{cccc}
0 & 1 & 0 & 0 \\
1 & 0 & \alpha (\gamma - \alpha ) & 0 \\
0 & 0 & k\nu  & k \\ 
0 & 0 & 0 & \alpha 
\end{array}
\right) .
\end{align}
Then we have
\begin{align}
& S^{-1} \tilde{A} ^{(0)} _{1}  S = \left(
\begin{array}{cccc}
-\nu & 0 & 0 & 0 \\
0 & 0 & -k\nu \alpha & 0 \\
0 & 0 & 0 & 0 \\ 
0 & 0 & 0 & 0 
\end{array}
\right) , \quad 
S^{-1} \tilde{A} ^{(1)} _{0}  S = 
\left(
\begin{array}{cccc}
\alpha -\gamma  & 0 & 0 & 0 \\
0 & 0 & 0 & 0 \\
1/\alpha  & 0 & 0 & 0 \\ 
-\nu/\alpha & 0 & 0 & \alpha 
\end{array}
\right) . \label{eq:S-1AS}
\end{align}
Since the second, the third and the fourth column of the matrix $S$ are divisors of the quotient space $mc_{\alpha } (\Cplx ^2 ) = (\Cplx ^2)^{\oplus 2} / ({\mathcal K} + {\mathcal L}(\alpha ) )$, the matrix elements of $\tilde{A} ^{(1)} _{0}  $ and $\tilde{A} ^{(0)} _{1} $ appear as $(1,1)$-elements of Eq.(\ref{eq:S-1AS}).
Hence Eq.(\ref{eq:dYdzA01zY0}) is transformed to 
\begin{align}
& \frac{dy}{dz} 
= \left( \nu +\frac{\alpha -\gamma }{z} \right) y
\label{eq:dydzrk1}
\end{align}
by the middle convolution $mc_{\alpha }$.
The solutions of Eq.(\ref{eq:dydzrk1}) is given by $y = c \exp (\nu z) z^{\alpha -\gamma } $ ($c$: a constant).

We are going to recover Eq.(\ref{eq:dYdzA01zY0}) from Eq.(\ref{eq:dydzrk1}) and obtain integral representations of solutions of Eq.(\ref{eq:dYdzA01zY0}), which arise from the equality $mc _{-\alpha } mc _{\alpha } =\rm{id}$.
We apply middle convolution $mc _{-\alpha }  $ to Eq.(\ref{eq:dydzrk1}).
Then we have
\begin{align}
& \frac{dW}{dz}=
\left(
-\left(
\begin{array}{cc}
-\nu  & \alpha -\gamma \\
 0 & 0 
\end{array}
\right) 
+ \frac{1}{z}
\left(
\begin{array}{cc}
0 & 0 \\
-\nu  & -\gamma 
\end{array}
\right) 
\right) W,
\label{eq:W} 
\end{align}
It follows from Theorem \ref{thm:DRintegrepr} that the function 
\begin{equation}
W = \left(
\begin{array}{l}
\int _{C} \exp (\nu w) w^{\alpha -\gamma } (z-w)^{-\alpha } dw \\
\int _{C} \exp (\nu w) w^{\alpha -\gamma -1} (z-w)^{-\alpha } dw 
\end{array}
\right) ,
\end{equation}
is a solution of Eq.(\ref{eq:W}) by choosing a cycle $C $ appropriately.
For simplicity we assume $\nu \in \Rea _{>0}$.
Then we can take cycles $C$ which start from $w=-\infty $, move along a real axis, turn the point $w= z$ or $w=0$ and come back to $w=-\infty  $.
By setting
\begin{equation}
W= \left(
\begin{array}{cc}
\alpha -\gamma & -k \\
\nu & 0
\end{array}
\right) \tilde{W} ,
\end{equation}
we recover the matrices in Eq.(\ref{eq:A0A1}) such as the function $\tilde{W}$ satisfies Eq.(\ref{eq:dYdzA01zY0}).
Consequently we obtain integral representations of solutions of Eq.(\ref{eq:dYdzA01zY0}) which are expressed as
\begin{equation}
\tilde{W} = \frac{1}{k\nu} 
\left(
\begin{array}{cc}
0 & k \\
-\nu & \alpha -\gamma 
\end{array}
\right)
\left(
\begin{array}{l}
\int _{C} \exp (\nu w) w^{\alpha -\gamma } (z-w)^{-\alpha } dw \\
\int _{C} \exp (\nu w) w^{\alpha -\gamma -1} (z-w)^{-\alpha } dw 
\end{array}
\right) . 
\end{equation}
In particular, the function
\begin{equation}
y(z)= \int _{C} \exp (\nu w) w^{\alpha -\gamma -1} (z-w)^{-\alpha } dw 
\end{equation}
satisfies Eq.(\ref{eq:y1DE}), and we obtain integral representations of solutions of the confluent hypergeometric differential equation.

Next we consider the case that $A^{(0)}_1$ is nilpotent.
Set
\begin{equation}
A^{(0)} _1 = \left(
\begin{array}{cc}
0 & -1 \\
0 & 0  
\end{array}
\right) , \quad A^{(1)} _0 = \left(
\begin{array}{cc}
a_{1,1} & a_{1,2} \\
a_{2,1} & a_{2,2}  
\end{array}
\right) .
\label{eq:Anilp}
\end{equation}
Then it follows from irreducibility that $a_{2,1} \neq 0$.
The index of rigidity is also two.
But we cannot reduce to rank one case by applying additions $A^{(0)} _1 \rightarrow A^{(0)} _1 + \alpha I_2$, $A^{(1)} _0 \rightarrow A^{(1)} _0 + \beta I_2$ and middle convolution $mc _{\mu } $, because ${\rm dim} ( {\mathcal L} '(\mu )) \leq 1$ for any $\alpha, \beta , \mu$, which follows from $a_{2,1} \neq 0$.
Note that solutions of the differential equations determined by Eq.(\ref{eq:Anilp}) are expressed in terms of Bessel's function by setting $z=x^2$.

\section{The case $m_i \leq 1$ for all $i$} \label{sec:mileq1}

In this section, we investigate the index of rigidity and middle convolution for the case $m_i \leq 1$ for all $i$ (see Eq.(\ref{eq:DEirr0})).
The case $m_i=0$ is included to the case $m_i=1$ by setting $A_1^{(i)} =0$.
We assume that $A_1^{(i)}$ is semi-simple for all $i$ and $V(=\Cplx ^n)$ is irreducible as $\langle {\bold A} \rangle $-module.

\subsection{Index of rigidity}
To study the index of rigidity, we investigate Eq.(\ref{eq:relrig}) for the case $m_i=1$.
We ignore the superscript ${}^{(i)}$.
Then Eq.(\ref{eq:relrig}) is written as
\begin{align}
& A_1 C_1 = C_1 A_1, \quad A_1 C_0 - C_0 A_1 + A_0 C_1 - C_1 A_0 =0 .
\label{eq:relrig1}
\end{align}
By the assumption that $A_1$ is semi-simple,
we diagonalize $A_1$ as
\begin{align}
& 
P ^{-1} A_1 P= 
\left.
\left( 
\begin{array}{cccc}
d_1 I_{n_1} & 0 & \dots & 0  \\
0 & d_2 I_{n_2} & \dots & 0  \\
0 & 0 & \ddots & \vdots \\
0 & \cdots  & 0 & d_{k} I_{n_{k}} 
\end{array}
\right)
\right. , \quad d_i \neq d_j \; (i\neq j). \label{eq:A1blk}
\end{align}
Write
\begin{align}
& 
 P ^{-1} A_0 P = 
\left.
\left( 
\begin{array}{cccc}
A_0^{[1,1]} & A_0^{[1,2]} & \dots & A_0^{[1,k]}   \\
A_0^{[2,1]} & A_0^{[2,2]} & \dots & A_0^{[2,k]}   \\
\vdots & \vdots & \ddots & \vdots \\
 A_0^{[k,1]} &  A_0^{[k,2]}& \dots & A_0^{[k,k]} 
\end{array}
\right)
\right. , \label{eq:A0blk}
\end{align}
where $A_0^{[i,j]}$ is a $n_i \times n_j$ matrix.
It follows from $A_1 C_1 = C_1 A_1 $ that $C_1$ is written as
\begin{align}
& P^{-1}C_1 P  = 
\left.
\left( 
\begin{array}{cccc}
C^{[1]}_1 & 0 & \dots & 0  \\
0 & C^{[2]}_1 & \dots & 0  \\
0 & 0 & \ddots & \vdots \\
0 & \cdots  & 0 & C^{[k]}_1 
\end{array}
\right)
\right. ,
\end{align}
where $ C^{[l]}_1 $ is a $n_l \times n_l$ matrix. 
Then
\begin{eqnarray}
& &  P^{-1}( A_0 C_1 - C_1 A_0 ) P= \\
& &  \left( 
\begin{array}{cccc}
A_0^{[1,1]} C_1^{[1]} - C_1^{[1]} A_0^{[1,1]} & A_0^{[1,2]} C_1^{[2]} - C_1^{[1]} A_0^{[1,2]} & \dots & A_0^{[1,k]} C_1^{[k]} - C_1^{[1]} A_0^{[1,k]}   \\
A_0^{[2,1]} C_1^{[1]} - C_1^{[2]} A_0^{[2,1]} & A_0^{[2,2]} C_1^{[2]} - C_1^{[2]} A_0^{[2,2]} & \dots & A_0^{[2,k]} C_1^{[k]} - C_1^{[2]} A_0^{[2,k]}   \\
\vdots & \vdots & \ddots & \vdots \\
A_0^{[k,1]} C_1^{[1]} - C_1^{[k]} A_0^{[k,1]} & A_0^{[k,2]} C_1^{[2]} - C_1^{[k]} A_0^{[k,2]} & \dots & A_0^{[k,k]} C_1^{[k]} - C_1^{[k]} A_0^{[k,k]}   \\
\end{array}
\right) .  \nonumber 
\end{eqnarray}
By writing
\begin{align}
&  P^{-1}C_0 P = 
\left.
\left( 
\begin{array}{cccc}
C_0^{[1,1]} & C_0^{[1,2]} & \dots & C_0^{[1,k]}   \\
C_0^{[2,1]} & C_0^{[2,2]} & \dots & C_0^{[2,k]}   \\
\vdots & \vdots & \ddots & \vdots \\
C_0^{[k,1]} & C_0^{[k,2]}& \dots & C_0^{[k,k]} 
\end{array}
\right)
\right. ,
\end{align}
we have
\begin{align}
& P^{-1}( A_1 C_0 - C_0 A_1 )  P = \\
& \left( 
\begin{array}{cccc}
0 & (d_1-d_2) C_0^{[1,2]} & \dots & (d_1-d_{k})  C_0^{[1,k]}   \\
(d_2-d_1) C_0^{[2,1]} & 0 & \dots & (d_2-d_{k}) C_0^{[2,k]}   \\
\vdots & \vdots & \ddots & \vdots \\
(d_{k} -d_1) C_0^{[k,1]} & (d_{k} -d_2) C_0^{[k,2]}& \dots & 0 
\end{array}
\right) . \nonumber
\end{align}
It follows from $A_1 C_0 - C_0 A_1 + A_0 C_1 - C_1 A_0 =0$ that $A_0^{[i,i]} C_1^{[i]} = C_1^{[i]} A_0^{[i,i]}$ and 
$C_0^{[i,j]} $ $(i\neq j)$ is determined as $C_0^{[i,j]}=-(A_0^{[i,j]} C_1^{[j]} - C_1^{[i]} A_0^{[i,j]})/(d_i-d_j)$.
Elements of $C_0^{[i,i]} $ are not restricted by relations.
Hence the dimension of solutions of Eq.(\ref{eq:relrig1}) is
\begin{align}
&  \sum _{l=1} ^k \{  (n_{l} )^2 + \dim  Z(A_0^{[l,l]} ) \} ,  
\end{align}
where $Z ( A_0^{[l,l]} )= \{ X \in \Cplx ^{n_l \times n_l} | X A_0^{[l,l]} = A_0^{[l,l]} X \}$.
Let $I_{a,b}$ be the $a \times b$ matrix whose $(i,j)$-element is given by $\delta _{i,j}$, ${\bold q} = (q_{1} ,\dots ,q_{p} ) \in \Zint ^{p}$ ($q_{1} +\dots +q_{p} =n $, $q_{1} \geq \dots \geq q_{p}\geq 1$), $\underline{\lambda } = (\lambda _{1},\dots , \lambda _{p} ) \in \Cplx ^{p}$.
Following Oshima \cite{O}, set
\begin{equation}
L({\bold q} ; \: \underline{\lambda } ) = 
\left(
\begin{array}{cccc}
\lambda _{1} I_{q_{1}} & I_{q_1,q_2} & 0 & \cdots \\
0 & \lambda _{2} I_{q_{2}} & I_{q_2,q_3}  & \ddots \\
0 & 0 & \lambda _{3} I_{q_{3}} & \ddots \\
\vdots & \ddots & \ddots & \ddots 
\end{array}
\right) .
\end{equation}
Every matrix is conjugate to $L({\bold q} ; \: \underline{\lambda } ) $ for some ${\bold q}$, $\underline{\lambda }$.
Note that if $\lambda  _i \neq \lambda _j$ $(i\neq j) $ then the matrix $L({\bold q} ; \: \underline{\lambda } ) $ is conjugate to the diagonal matrix whose multiplicity of the eigenvalue $\lambda _i$ is  $q_i$.
If the matrix $ A_0^{[l,l]}$ is conjugate to $L(n_{l,1}, \dots , n_{l,p_{l}} ; \: d_{l,1} , \dots, d_{l,p_{l}} ) $, then 
the dimension of solutions of Eq.(\ref{eq:relrig1}) is given by
\begin{align}
& \sum _{l=1} ^k \left\{  (n_{l} )^2  + \sum _{j=1}^{p_l} (n_{l,j} )^2 \right\} .
\label{eq:dimC1C0}
\end{align}

We denote the type of multiplicities of the matrices $(A_1, A_0)$ which are expressed as Eqs.(\ref{eq:A1blk}), (\ref{eq:A0blk}) and $A _0 ^{[l,l]} $ $(l=1,\dots ,k)$ is conjugate to $ L (n_{l,1} ,\dots ,n_{l,p_l} ; \lambda _{l,1} , \dots ,\lambda _{l,p_l})$ by
\begin{equation}
(n_1, n_2, \dots ,n_k) - ((n_{1,1}, \dots ,n_{1,p_1}) , (n_{2,1},\dots ,n_{2,p_2}), \dots , (n_{k,1},\dots ,n_{k,p_k})). 
\end{equation}
Note that $n_l =n_{l,1} + \dots +n_{l,p_l}$ $(l=1,\dots ,k)$.
Then the local index of rigidity of the matrices $(A_1, A_0)$ is calculated as
\begin{align}
2n^2 - \sum _{l=1} ^k \left\{  (n_{l} )^2  + \sum _{j=1}^{p_l} (n_{l,j} )^2 \right\} .
\label{eq:locidxA1A0}
\end{align}
If $A_1 =0$, then $k=1$, $n_1=n$ and we simplify the notation $(n_1) - ((n_{1,1}, \dots ,n_{1,p_1})) $ by $(n_{1,1}, \dots ,n_{1,p_1})$.
Note that the notation $(n_{1,1}, \dots ,n_{1,p_1})$ was already adapted by Kostov \cite{Kos} and Oshima \cite{O} for the case of regular singularity.

By combining Eq.(\ref{eq:locidxA1A0}) with Eq.(\ref{eq:idxlidx}) we have the following proposition:
\begin{proposition}
We assume that $m_0=\dots =m_r=1$, $A_1^{(i)}$ are semi-simple for $i=0,\dots ,r$ and $\Cplx ^n$ is irreducible as $\langle {\bold A} \rangle $-module.
Let 
\begin{align}
& (n^{(i)}_1, n^{(i)}_2, \dots ,n^{(i)}_{k^{(i)}}) - \\
& ((n^{(i)}_{1,1}, \dots ,n^{(i)}_{1,p^{(i)}_1}) , (n^{(i)}_{2,1},\dots ,n^{(i)}_{2,p_2^{(i)}}), \dots , (n^{(i)}_{k^{(i)},1},\dots ,n^{(i)}_{k^{(i)},p^{(i)}_{k^{(i)}}})), \nonumber
\end{align}
$(n^{(i)}_1 \geq n^{(i)}_2 \geq \dots \geq n^{(i)}_{k^{(i)}}$, $n^{(i)}_{j,1}\geq \dots \geq n^{(i)}_{j,p_j^{(i)}} \; (j=1,\dots , k^{(i)})) $ be the type of multiplicities of $(A^{(i)}_1, A^{(i)}_0)$.
Then the index of rigidity is equal to
\begin{align}
& {\rm idx}({\bold A}) = \sum _{i=0}^r  \sum _{j=1}^{k^{(i)}} \left( (n^{(i)}_{j})^2 +  \sum _{j'=1} ^{ p_j^{(i)}} (n^{(i)}_{j,j'})^2 \right) -2 r n^2 .
\label{eq:idxrig}
\end{align}
\end{proposition}

\subsection{Subspace}
Next we investigate solutions of the equations
\begin{align}
& A_1 v_0 =0, \; A_1 v_1 +A_0 v_0= 0,
\label{eq:A1A0v}
\end{align}
for the case that the matrices $A_1$ and $A_0$ are expressed as Eqs.(\ref{eq:A1blk}), (\ref{eq:A0blk}) and $d_1=0$ to understand the subspaces ${\mathcal K}^{(i)}$ and ${\mathcal L}'{(\lambda )}$.
Write
\begin{align}
v_1= P^{-1} \left(
\begin{array}{c}
 v_1^{[1]} \\
\vdots \\
v_1^{[n]} 
\end{array}
\right) , 
\; v_0 = P^{-1} \left(
\begin{array}{c}
v_0^{[1]} \\
\vdots \\
v_0^{[n]} 
\end{array}
\right) ,\;
(v_1^{[l]}, v_0^{[l]} \in \Cplx ^{n_l} ).
\end{align}
It follows from $A_1 v_0 =0$ that $v_{0} ^{[i]} =0$ $(i \geq 2 )$.
Hence 
\begin{align}
P^{-1} ( A_0 v_0 +A_1 v_1)  =\left(
\begin{array}{c}
A _0 ^{[1,1]} v_{0}^{[1]}  \\
A _0 ^{[1,2]} v_{0}^{[1]} + d_2 v_{1}^{[2]} \\
\vdots \\
A_0 ^{[1,k]} v_{0} ^{[1]} + d_k v_{1}^{[k]} 
\end{array}
\right) =0 , 
\end{align}
$v_{0} ^{[1]} \in {\rm Ker} ( A _0 ^{[1,1]} )$ and $v_{1}^{[l]} $ $(l\geq 2)$ is determined as $-A_0 ^{[1,l]} v_{0} ^{[l]} /d_l$.
The elements of $v_{1}^{[1]} $ are independent.
Hence the dimension of solutions of Eq.(\ref{eq:A1A0v}) is
$n_1 + \dim ({\rm Ker} ( A _0 ^{[1,1]} ) )$.
If the matrix $ A_0^{[1,1]}$ is conjugate to $L(n_{1,1}, \dots , n_{1,p_{1}} ; \: d_{1,1} , \dots, d_{1,p_{1}} ) $ and $d_{1,1} =0$, then 
the dimension of solutions of Eq.(\ref{eq:A1A0v})  
is $n_1 +n_{1,1}$.
\subsection{Middle convolution}
We now study the matrices for which the middle convolution is applied. 
Let $i \in \{ 1,\dots ,r \}$ and $m_i =1$.
We investigate the matrices $\tilde{A}^{(i)}_1$ and $\tilde{A}^{(i)}_0 $ for the case that the matrices $A^{(i)}_1$ and $A^{(i)}_0$ are expressed as Eqs.(\ref{eq:A1blk}), (\ref{eq:A0blk}), $d_1=0$ and $\Cplx ^n $ is irreducible as $\langle {\bold A} \rangle $-module.
By changing the order of the direct sum $V'= (\Cplx ^n)^{\oplus M} $, the matrices $\tilde{A}^{(i)}_1$ and $\tilde{A}^{(i)}_0 $ are expressed as
\begin{align}
& (QP^{\oplus M})^{-1} \tilde{A}_1 ^{(i)} QP^{\oplus M} =\left( 
\begin{array}{ccc}
P^{-1}A_1 P & P^{-1 } A_0P  +\mu I_n & P ^{-1} \overline{A} P^{\oplus (M-2)} \\
0 & 0 & 0 \\
0 & 0 & 0
\end{array}
\right) , \\
& (QP^{\oplus M})^{-1} \tilde{A}_0 ^{(i)}QP^{\oplus M} =\left( 
\begin{array}{ccc}
\mu & 0 & 0 \\
P^{-1}A_1 P & P^{-1 } A_0P  +\mu I_n  & P^{-1} \overline{A} P^{\oplus (M-2)} \\
0 & 0 & 0
\end{array}
\right) , \nonumber 
\end{align}
where $\overline{A} =\left( A^{(0)} _{m_0}  \dots A^{(i-1)} _{\delta _{i,1}} \; A^{(i+1)} _{m_{i+1}} \dots \right) $ and $Q$ represents the change of the order of the direct sum $(\Cplx ^n)^{\oplus M} $.
Set 
\begin{align}
& A_1^{\#} = P \left( 
\begin{array}{cccc}
0 \cdot I_{n_1} & 0 & \dots & 0  \\
0 & (d_2)^{-1} I_{n_2} & \dots & 0  \\
0 & 0 & \ddots & \vdots \\
0 & \cdots  & 0 & (d_{k})^{-1} I_{n_{k}} 
\end{array}
\right) P^{-1}, \\
& X= - P ^{-1} A_1^{\#} \left( A_0  +\mu I_n \; \overline{A} \right) P^{\oplus (M-1)}, \nonumber \\
& R= 
-  P ^{-1} A_1^{\#} ( A_0  +\mu I_n )P. \nonumber
\end{align}
Since the $j-$th row blocks of the matrix $P^{-1} (A_1 A_1^{\#} -I_n ) P $ zero for $j \geq 2$, the $j-$th row blocks of the matrix $X'=  P^{-1}A_1 P X + P^{-1 } \left( A_0  +\mu I_n \; \overline{A} \right) P^{\oplus (M-1)}$ and $P^{-1}A_1 P R +  P^{-1 } (A_0  +\mu I_n )P $ are also zero for $j \geq 2$.
We denote the size of the matrices $\tilde{A}_ 1^{(i)}$, $\tilde{A}_0 ^{(i)}$ on the space ${\mathcal M}= mc _{\mu } (\Cplx ^n ) $ by $\tilde{n} (= nM-\sum _{i=1}^r \dim {\mathcal K}^{(i)} - \dim {\mathcal L} (\mu ) )$.
Restrictions of the matrices  $\tilde{A}^{(i)}_1$ and $\tilde{A}^{(i)}_0 $ to the space ${\mathcal K}^{(j)}$ $(j\neq i)$ and ${\mathcal L} (\mu )$ are zero, which follow from the definitions of the spaces. 
Thus the matrix $\tilde{A} ^{(i)} _1$ on ${\mathcal M}$ is diagonalized as
\begin{align}
& \left. \left( 
\begin{array}{cc}
I & X \\
0 & I 
\end{array}
\right) ^{-1}
(QP^{\oplus M})^{-1} \tilde{A}_1 ^{(i)} QP^{\oplus M}
 \left( 
\begin{array}{cc}
I & X \\
0 & I 
\end{array}
\right) \right| _{{\mathcal M}'} \\
& = \left. \left( 
\begin{array}{cc}
P^{-1}A_1 P  & X' \\
0 & 0 
\end{array}
\right)  \right| _{{\mathcal M}'} \nonumber 
= \left( 
\begin{array}{cccc}
d_2 I_{n_2} & 0 & \cdots & 0 \\
0  & \ddots &  \ddots & 0 \\
\vdots & \ddots & d_{k} I_{n_{k}} & 0 \\
0 & \cdots & 0 & 0
\end{array}
\right) ,
\end{align}
where ${\mathcal M}'= \left( 
\begin{array}{cc}
I & X \\
0 & I 
\end{array}
\right) ^{-1}
(QP^{\oplus M})^{-1} {\mathcal M}
 $, $I$ is a unit matrix of the suitable size 
and the dimension of the kernel of $ \tilde{A} ^{(i)} _1 | _{\mathcal M} $ is $\tilde{n} -n + n_1$.
Since the matrix $P^{-1}A_1 P $ is diagonal, the $[l,l]$-block of $R P^{-1}A_1 P$ coincides with that of $ P^{-1}A_1 P R$, and it is equal to the $[l,l]$-block of $-P^{-1 } (A_0  +\mu I_n )P $, if $l \geq 2$.
Hence 
\begin{align}
& \left. \left( 
\begin{array}{cc}
I & X \\
0 & I 
\end{array}
\right) ^{-1}
(QP^{\oplus M})^{-1} \tilde{A}_0 ^{(i)} QP^{\oplus M}
 \left( 
\begin{array}{cc}
I & X \\
0 & I 
\end{array}
\right) \right| _{{\mathcal M}'} \\
& = \left. \left( 
\begin{array}{cc}
\mu - R P^{-1}A_1 P	 & \mu X -XX' \\
 P^{-1}A_1P  & X'
\end{array}
\right)  \right| _{{\mathcal M}'} \nonumber 
\\
& = \left( 
\begin{array}{cccc}
A_0^{[2,2]} +2\mu I_{n_2} & * & \cdots  &* \\
*  & \ddots &  \ddots & * \\
\vdots & * & A_0^{[k,k]} +2\mu I_{n_{k}} & * \\
\vdots & \ddots & * & \overline{X}' 
\end{array}
\right) , \nonumber
\end{align}
where $\overline{X}' $ is expressed as 
\begin{align}
&  \overline{X}' =  \left( 
\begin{array}{cc}
\overline{A }_0 ^{[1,1]} + \mu I  &  \overline{X}''  \\
0  & 0 
\end{array} 
\right) , \nonumber
\end{align}
$\overline{A }_0 ^{[1,1]}$ is the matrix obtained by replacing the domain and the range of $A _0 ^{[1,1]} $ to $\Cplx ^{n_1} / \mbox{Ker} A _0 ^{[1,1]}$.
It follows from irreducibility that
the rank of $\overline{X}' $ is equal to the size of $\overline{A }_0 ^{[1,1]}$.
Thus, if the matrix $A _0 ^{[1,1]} $ is conjugate to 
\begin{align}
& L(n _{1}^{\langle 0 \rangle } ,\dots , n _{p _{\langle 0 \rangle  }}^{\langle 0 \rangle };  0,\dots ,0) \oplus L (n _{1}^{\langle -\mu  \rangle } ,\dots , n _{p _{\langle -\mu \rangle  }}^{\langle -\mu \rangle };  -\mu  ,\dots ,-\mu  ) \\
&  \oplus L (m_{1,1} ,\dots ,m_{1,p'_1} ; \lambda _{1,1} , \dots ,\lambda _{1,p'_1}) \qquad (\lambda _{1,j} \neq 0, -\mu ), \nonumber
\end{align}
then we have $\tilde{n} - n +n _{1}^{\langle 0 \rangle } \geq n _{1}^{\langle -\mu  \rangle }$ and the matrix $\overline{X}' $ is conjugate to 
\begin{align}
& L (\tilde{n} - n +n _{1}^{\langle 0 \rangle } , n _{1}^{\langle -\mu  \rangle } ,\dots , n _{p _{\langle -\mu \rangle  }}^{\langle -\mu \rangle } ; 0 ,\dots ,0  )  \\
& \oplus L(n _{2}^{\langle 0 \rangle } ,\dots , n _{p _{\langle 0 \rangle  }}^{\langle 0 \rangle }; \mu ,\dots ,\mu ) \oplus L (m_{1,1} ,\dots ,m_{1,p'_1} ; \lambda _{1,1} +\mu , \dots ,\lambda _{1,p'_1} +\mu ). \nonumber
\end{align}
If the matrix $A _0 ^{[l,l]} $ $(l\geq 2) $ is conjugate to $ L (m_{l,1} ,\dots ,m_{l,p_l} ; \lambda _{l,1} , \dots ,\lambda _{l,p_l})$, then the matrix $A _0 ^{[l,l]} +2\mu I_{n_l } $ is conjugate to $ L (m_{l,1} ,\dots ,m_{l,p_l}  ; \lambda _{l,1} +2\mu , \dots ,\lambda _{l,p_l} +2\mu)$.
Hence
\begin{align}
& {\rm idx} _i(mc _{\mu }({\bold A })) - {\rm idx} _i ({\bold A}) = (\tilde{n} -n+ n_1)^2 +\sum _{j=2}^k (n_j)^2  \label{eq:idximcAA} \\
&  + (\dim Z(\overline{A }_0 ^{[1,1]} +\mu ))^2 + (\tilde{n} -n+n_{1,1})^2 +\sum _{j=2}^k (\dim Z(A _0 ^{[j,j]} +2\mu ))^2  -2\tilde{n}^2 \nonumber \\
& - \left( \sum _{j=1}^k (n_j)^2 + (\dim Z(\overline{A }_0 ^{[1,1]} ))^2 + n_{1,1}^2 + \sum _{j=2}^k (\dim Z(A _0 ^{[j,j]} ))^2 -2n^2 \right) \nonumber \\
& = 2(n -\tilde{n})(2n-n_1 -n_{1,1}) = 2(n -\tilde{n})(2n-\dim {\mathcal K} ^{(i)} ) , \nonumber
\end{align}
where $n_{1,1} = \dim ({\rm Ker} A _0 ^{[1,1]} )$.

We investigate the matrices $\tilde{A}^{(0)}_1$ and $\tilde{A}^{(0)}_0 $ for the case that $m_0 =1$, the matrices $A^{(0)}_1$ and $A^{(0)}_0$ are expressed as Eqs.(\ref{eq:A1blk}), (\ref{eq:A0blk}) and $d_1=0 $.
By changing the order of the direct sum $V'= (\Cplx ^n)^{\oplus M} $, 
the matrices $\tilde{A}^{(0)}_1$ and $\tilde{A}^{(0)}_0 + \mu I$ are simultaneously conjugate to
\begin{align}
& \tilde{A}_1 ^{(0)} \sim \left( 
\begin{array}{ccc}
P^{-1}A_1 P & P^{-1 } A_0P  & P ^{-1} \overline{A} P^{\oplus (M-2)} \\
0 & 0 & 0 \\
0 & 0 & 0
\end{array}
\right) , \\
& \tilde{A}_0 ^{(0)} + \mu I \sim  \left( 
\begin{array}{ccc}
\mu & 0 & 0 \\
P^{-1}A_1 P & P^{-1 } A_0P  & P^{-1} \overline{A} P^{\oplus (M-2)} \\
0 & 0 & 0
\end{array}
\right) , \nonumber 
\end{align}
where $\overline{A} =\left( A^{(1)} _{m_1}  \dots A^{(1)} _{1} \; A^{(2)} _{m_{2}} \dots \right) $.
It follows from similar argument to the case $\tilde{A}^{(i)}_1$, $\tilde{A}^{(i)}_0 $ $(i\neq 0)$ that $\tilde{A}^{(0)}_1$ and $\tilde{A}^{(0)}_0 $ are simultaneously conjugate to
\begin{align}
& \left. \tilde{A}_1 ^{(0)}  \right| _{{\mathcal M}}  \sim \left( 
\begin{array}{cccc}
d_2 I_{n_2} & 0 & \cdots & 0 \\
0  & \ddots &  \ddots & 0 \\
\vdots & \ddots & d_{k} I_{n_{k}} & 0 \\
0 & \cdots & 0 & 0
\end{array}
\right) , \\
& \left. \tilde{A}_0 ^{(0)}  \right| _{{\mathcal M}}  \sim 
\left( 
\begin{array}{cccc}
A_0^{[2,2]} & * & \cdots  &* \\
*  & \ddots &  \ddots & * \\
\vdots & * & A_0^{[k,k]} & * \\
\vdots & \ddots & * & \overline{X}' 
\end{array}
\right) ,
\end{align}
where $\overline{X}' $ is expressed as 
\begin{align}
&  \overline{X}' = \left( 
\begin{array}{cc}
\overline{A }_0 ^{[1,1]} - \mu I \! &  \overline{X}''  \\
0  & - \mu I  
\end{array} 
\right) , \nonumber
\end{align}
$\overline{A }_0 ^{[1,1]}$ is the matrix obtained by replacing the domain and the range of $A _0 ^{[1,1]} $ to $\Cplx ^{n_1} / \mbox{Ker}(A _0 ^{[1,1]} -\mu I)$.
If the matrix $A _0 ^{[1,1]} $ is conjugate to 
\begin{align}
& L(n _{1}^{\langle \mu \rangle } ,\dots , n _{p _{\langle \mu \rangle  }}^{\langle \mu \rangle }; \mu ,\dots , \mu ) \oplus L (n _{1}^{\langle 0 \rangle } ,\dots , n _{p _{\langle 0 \rangle  }}^{\langle 0 \rangle }; 0 ,\dots ,0  )  \\
& \oplus L (m_{1,1} ,\dots ,m_{1,p'_1} ; \lambda _{1,1}  , \dots ,\lambda _{1,p'_1} ), \qquad (\lambda _{1,j} \neq 0, \mu ), \nonumber
\end{align}
then we have $ \tilde{n} - n+ n _{1}^{\langle \mu \rangle } \geq n _{1}^{\langle 0 \rangle }$ and the matrix $\overline{X}' $ is conjugate to 
\begin{align}
& L (\tilde{n} - n +n _{1}^{\langle \mu \rangle } , n _{1}^{\langle 0 \rangle } ,\dots , n _{p _{\langle 0 \rangle  }}^{\langle 0 \rangle } ; -\mu ,\dots ,- \mu )  \\
& \oplus L(n _{2}^{\langle \mu \rangle } ,\dots , n _{p _{\langle \mu \rangle  }}^{\langle \mu \rangle }; 0 ,\dots ,0 ) \oplus L (m_{1,1} ,\dots ,m_{1,p'_1} ; \lambda _{1,1} - \mu , \dots ,\lambda _{1,p'_1} -\mu ). \nonumber
\end{align}
%
Hence we also have
\begin{align}
& {\rm idx} _0(mc _{\mu }({\bold A })) - {\rm idx} _0 ({\bold A}) = 2(n -\tilde{n})(2n-\dim {\mathcal L} (\mu ) ) . \label{eq:idx0mcAA} 
\end{align}
We include the case $m_i=0$ to the case $m_i=1$ by setting $A^{(i)} _1=0$ and we have
$\tilde{n} = n(2r+1) -\sum _{i=1}^r \dim {\mathcal K}^{(i)} - \dim {\mathcal L} (\mu ) $.
It follows from Eqs.(\ref{eq:idximcAA}), (\ref{eq:idx0mcAA}) that
\begin{align}
& {\rm idx} (mc _{\mu }({\bold A })) - {\rm idx} ({\bold A}) = 2\tilde{n}^2 -2n^2 + \sum _{i=0}^r \left\{ {\rm idx} _i(mc _{\mu }({\bold A })) - {\rm idx} _i ({\bold A}) \right\} \\
& = 2\tilde{n}^2 -2n^2 + 2(n -\tilde{n}) \left( 2(r+1)n -\sum _{i=1}^r \dim {\mathcal K}^{(i)} - \dim {\mathcal L} (\mu ) \right) \nonumber \\
& = 2\tilde{n}^2 -2n^2 + 2(n -\tilde{n})(n +\tilde{n}) =0 . \nonumber
\end{align}
Hence the index of rigidity is preserved by application of middle convolution, i.e. ${\rm idx}(mc _{\mu }({\bold A })) = {\rm idx}({\bold A})$.
Namely we obtain the following proposition.
\begin{proposition} \label{prop:indrig}
If $m_i \leq 1$, the matrices $A^{(i)}_1$ are semi-simple for all $i$ and $\langle {\bold A} \rangle$ is irreducible, then the index of rigidity is preserved by application of middle convolution, i.e. ${\rm idx}(mc _{\mu }({\bold A })) = {\rm idx}({\bold A})$ for all $\mu \in \Cplx$.
\end{proposition}

\subsection{Classification}
\begin{proposition} \label{prop:classif}
Assume that $m_i \leq 1$ $(i=0,\dots ,r)$, $A^{(i)}_1 $ are semi-simple for all $i$ and $\langle {\bold A} \rangle$ is irreducible.
We identify the case $m_i=0$ with the case $m_i=1$ and $A^{(i)}_1 =0$.\\
(i) If ${\rm idx}({\bold A}) =2$, then ${\bold A} $ is transformed to the rank one matrices by applying addition and middle convolution repeatedly.\\
(ii) If ${\rm idx}({\bold A}) =0$, then ${\bold A} $ is transformed to one of the following cases by applying middle convolution and addition repeatedly, where $d \in \Zint _{\geq 1}$.
\begin{align}
\mbox{Four singularities} : & \quad \{ (d,d) , \; (d, d), \; (d,d) , \; (d,d) \} ,\\
\mbox{Three singularities} :  & \quad \{ (d,d,d) , \; (d, d, d), \; (d,d,d) \} , \nonumber \\
& \quad \{ (2d,2d) , \; (d,d,d,d) , \; (d, d, d ,d)\} ,  \nonumber \\
& \quad \{ (3d,3d), \; (2d,2d,2d) , \; (d,d, d,d, d,d)\} , \nonumber \\
& \quad \{  (d,d) -((d), (d)), \; (d,d) , \; (d,d) \} , \nonumber 
\end{align}
\begin{align}
\mbox{Two singularities} :& \quad \{  (d,d) -((d), (d)), \; (d,d) -((d), (d)) \} ,\nonumber \\
& \quad \{ (d,d,d) -((d), (d), (d) ), \; (d,d,d)  \} ,\nonumber \\
& \quad \{ (d,d,d,d) -((d), (d), (d) ,(d)), \; (2d,2d) \} , \nonumber \\
& \quad \{ (2d,2d) -((d,d), (d,d)), \; (d,d,d,d) \} , \nonumber \\
& \quad \{ (3d,2d) -((d,d,d), (2d)), \; (d,d,d,d,d) \} , \nonumber \\
& \quad \{ (2d,2d,2d) -((d,d), (d,d), (d,d)), \; (3d,3d) \} , \nonumber \\
& \quad \{ (3d,3d,2d) -((d,d,d), (d,d,d), (2d)), \; (4d,4d) \} , \nonumber \\
& \quad \{ (5d,4d,3d) -((d,d,d,d,d), (2d,2d), (3d)), \; (6d,6d) \} , \nonumber \\
& \quad \{ (5d,4d) -((d,d,d,d,d), (2d,2d)), \; (3d,3d,3d) \} , \nonumber \\
& \quad \{ (3d,3d) -((d,d,d), (d,d,d)), \; (2d,2d,2d) \} ,  \nonumber \\
& \quad \{ (5d,3d) -((d,d,d,d,d), (3d)), \; (2d,2d,2d,2d)  \} ,\nonumber \\
& \quad \{ (4d,3d) -((2d,2d), (3d)), \; (d,d,d,d,d,d,d) \} . \nonumber 
\end{align}
\end{proposition}
\begin{proof}
(i) It is enough to show that the size of matrices can be decreased by appropriate application of addition and middle convolution, because the size of matrices is reduced to one by applying many times.

Let 
\begin{align}
& (n^{(i)}_1, n^{(i)}_2, \dots ,n^{(i)}_{k^{(i)}}) - \\
& ((n^{(i)}_{1,1}, \dots ,n^{(i)}_{1,p^{(i)}_1}) , (n^{(i)}_{2,1},\dots ,n^{(i)}_{2,p_2^{(i)}}), \dots , (n^{(i)}_{k^{(i)},1},\dots ,n^{(i)}_{k^{(i)},p^{(i)}_{k^{(i)}}})) ,\nonumber
\end{align}
$ ( n^{(i)}_1 \geq n^{(i)}_2 \geq \dots \geq n^{(i)}_{k^{(i)}}$, $n^{(i)}_{j,1}\geq \dots \geq n^{(i)}_{j,p_j^{(i)}} )$ be the type of multiplicities of $(A^{(i)}_1, A^{(i)}_0)$.
Note that  $ n^{(i)}_{j,1} + \dots + n^{(i)}_{j,p_j^{(i)}}= n^{(i)}_j$.
Then
\begin{equation}
(n^{(i)}_{j})^2 + \sum _{j'} (n^{(i)}_{j,j'})^2 \leq (n^{(i)}_{j})^2 + n^{(i)}_{j,1} \sum _{j'} n^{(i)}_{j,j'} = n^{(i)}_{j} (n^{(i)}_{j} +n^{(i)}_{j,1}).
\label{eq:ineq01}
\end{equation}
There exists a number $l^{(i)}$ such that 
\begin{equation}
\frac{n}{n^{(i)}_{l^{(i)}}} \left( (n^{(i)}_{l^{(i)}})^2 +\sum _{j} (n^{(i)}_{l^{(i)},j})^2 \right)  \geq \sum _j (n^{(i)}_{j})^2 +  \sum _j \sum _{j'}  (n^{(i)}_{j,j'})^2,
\label{eq:nili2ineq}
\end{equation}
because if not we have contradiction as
\begin{align}
& \sum _{l} \left( (n^{(i)}_{l})^2 +  \sum _{j'}  (n^{(i)}_{l,j'})^2  \right) 
 \\
& < \sum _{l} \frac{n^{(i)}_{l}}{n} \left\{ \sum _j (n^{(i)}_{j})^2 +  \sum _j \sum _{j'}  (n^{(i)}_{j,j'})^2 \right\} 
= \sum _j (n^{(i)}_{j})^2 +  \sum _j \sum _{j'}  (n^{(i)}_{j,j'})^2. \nonumber
\end{align}
By combining with Eq.(\ref{eq:ineq01}),
we have 
\begin{equation}
n(n^{(i)}_{l^{(i)}} +n^{(i)}_{l^{(i)},1}) \geq \sum _j (n^{(i)}_{j})^2 +  \sum _j \sum _{j'}  (n^{(i)}_{j,j'})^2 .
\label{eq:ineqnnili}
\end{equation}

Recall that the index of rigidity is calculated as 
\begin{align}
& {\rm idx}({\bold A}) = \sum _{i=0}^r \left( \sum _j (n^{(i)}_{j})^2 +  \sum _j \sum _{j'}  (n^{(i)}_{j,j'})^2 \right) - 2rn^2 .
\label{eq:idxrig0}
\end{align}
If $\sum _{i=0}^r n^{(i)}_{l^{(i)}} +n^{(i)}_{l^{(i)},1} \leq 2rn$, then 
\begin{align}
& \sum _{i=0}^r \left( \sum _j (n^{(i)}_{j})^2 +  \sum _j \sum _{j'}  (n^{(i)}_{j,j'})^2 \right) 
\leq \sum _{i=0}^r n(n^{(i)}_{l^{(i)}} +n^{(i)}_{l^{(i)},1}) \leq 2 r n^2 , \label{eq:ineqidx}
\end{align}
which contradicts to $ {\rm idx}({\bold A}) =2$.
Hence 
\begin{equation}
\sum _{i=0}^r ( n^{(i)}_{l^{(i)}} +n^{(i)}_{l^{(i)},1} ) \geq 2rn +1.
\label{eq:idx2rn+1}
\end{equation}

We apply addition in order that $\dim {\mathcal K}^{(i)} = n^{(i)}_{l^{(i)}}+n^{(i)}_{l^{(i)},1} $ for $i=1,\dots ,r$
and the dimension of the kernel of $A^{(0)}_1$ is $n^{(0)}_{l^{(0)}}$.
Let $\mu $ be the value such that $\dim {\mathcal L} (\mu ) = n^{(0)}_{l^{(0)}}+n^{(0)}_{l^{(0)},1} $.
If $\mu =0$, then it follows from Eq.(\ref{eq:idx2rn+1}) that $\dim {\mathcal K} + \dim {\mathcal L} ' (0) - n ( M-1)= \sum _{i=0}^r (n^{(i)}_{l^{(i)}} +n^{(i)}_{l^{(i)},1} ) -2nr < 0$ and it contradicts to Proposition \ref{prop:relKL} (iii).
Thus $\mu \neq 0$, the size of matrices obtained by middle convolution is
\begin{equation}
(2r+1) n - \sum _{i=0}^r (n^{(i)}_{l^{(i)}} +n^{(i)}_{l^{(i)},1} ) ,
\label{eq:sizem}
\end{equation}
and it is no more than $n-1$, which follows from Eq.(\ref{eq:idx2rn+1}).
Hence the size can be decreased by addition and middle convolution.

(ii) 
It is sufficient to consider the case that the size of matrices cannot be decreased by addition and middle convolution.
Let ${l^{(i)}}$ be the number which satisfies Eq.(\ref{eq:nili2ineq}).
If $\sum _{i=0}^r (n^{(i)}_{l^{(i)}} +n^{(i)}_{l^{(i)},1} ) > 2rn$, then the size is decreased by middle convolution (see Eq.(\ref{eq:sizem})) or the system is reducible (the case $\mu =0 $).
Hence $\sum _{i=0}^r (n^{(i)}_{l^{(i)}} +n^{(i)}_{l^{(i)},1} ) \leq 2rn$.
If $\sum _{i=0}^r (n^{(i)}_{l^{(i)}} +n^{(i)}_{l^{(i)},1} ) < 2rn$, then we have $\sum _{i=0}^r \left( \sum _j (n^{(i)}_{j})^2 +  \sum _j \sum _{j'}  (n^{(i)}_{j,j'})^2 \right) < 2 r n^2 $ as Eq.(\ref{eq:ineqidx}) and it contradicts to $ {\rm idx}({\bold A}) =0$.
Thus $\sum _{i=0}^r (n^{(i)}_{l^{(i)}} +n^{(i)}_{l^{(i)},1} ) = 2rn$.
Since $ {\rm idx}({\bold A}) =0$, all inequalities in Eqs.(\ref{eq:ineqnnili}), (\ref{eq:ineqidx}) are equalities.
In particular we have 
$(n^{(i)}_{l^{(i)}})^2 + \sum _{j'} (n^{(i)}_{l^{(i)},j'})^2 = n^{(i)}_{l^{(i)}}(n^{(i)}_{l^{(i)}} +n^{(i)}_{l^{(i)},1}) $, which leads to $n^{(i)}_{l^{(i)},j'}  =n^{(i)}_{l^{(i)},1}$ for $1\leq j'\leq p^{(i)} _{l^{(i)}}$.
Then $((n^{(i)}_{l^{(i)}})^2 +\sum _{j} (n^{(i)}_{l^{(i)},j})^2 ) n/ n^{(i)}_{l^{(i)}} = n (n^{(i)}_{l^{(i)}} +n^{(i)}_{l^{(i)},1} ) = \sum _j (n^{(i)}_{j})^2 +  \sum _j \sum _{j'}  (n^{(i)}_{j,j'})^2 $.
Hence if Eq.(\ref{eq:nili2ineq}) is satisfied, then the inequality in Eq.(\ref{eq:nili2ineq}) is replaced by equality.
Therefore $n((n^{(i)}_{l})^2 +\sum _{j} (n^{(i)}_{l,j})^2 ) \leq n^{(i)}_{l} (\sum _j (n^{(i)}_{j})^2 +  \sum _j \sum _{j'}  (n^{(i)}_{j,j'})^2) $ for all $l$.
It follows from summing up with respect to $l$ that $((n^{(i)}_{l})^2 +\sum _{j} (n^{(i)}_{l,j})^2 ) n/ n^{(i)}_{l} = \sum _j (n^{(i)}_{j})^2 +  \sum _j \sum _{j'}  (n^{(i)}_{j,j'})^2 $ for all $l$.
By setting $l=l^{(i)}$ and repeating the discussion above, we have
\begin{align}
& n^{(i)}_{l,j}  =n^{(i)}_{l,1} = \frac{n^{(i)}_l}{ p^{(i)}_l} , \quad n^{(i)}_{l,1} ( p^{(i)}_l+ 1)= n^{(i)}_{1,1} ( p^{(i)}_1+ 1) ,
\end{align}
for all $j \in \{ 1,\dots , p^{(i)}_l\} $ and $l \in \{ 1,\dots ,k^{(i)} \}$.
Since $\sum _{i=0}^r (n^{(i)}_{1} +n^{(i)}_{1,1} ) = 2rn$, we have 
\begin{equation}
\sum _{i=0}^r \left( 2- \frac{n^{(i)}_{1} +n^{(i)}_{1,1}}{n} \right) =2.
\label{eq:2-nn2}
\end{equation}
Hence $r \geq 1$.

If $n^{(i)}_{1}=n$ and $p^{(i)}_{1}=1$, then
the matrices $A^{(i)}_1$ and $A^{(i)}_0$ are scalar, and we may omit the singularity corresponding to $A^{(i)}_1$ and $A^{(i)}_0$, because they are transformed to $A^{(i)}_1 =A^{(i)}_0 =0$ by applying addition.
If $n^{(i)}_{1}=n$ and $p^{(i)}_{1}\geq 2$, then we have $k ^{(i)} = 1$, $ n^{(i)}_{1,1} = \dots = n^{(i)}_{1,p^{(i)}_{1}} $ and $2- (n^{(i)}_{1} +n^{(i)}_{1,1})/{n} = 1-1/p^{(i)}_{1}$.
Thus $1/2\leq 2- (n^{(i)}_{1} +n^{(i)}_{1,1})/{n} <1$ and the equality holds iff $p^{(i)}_{1} =2$.
Since the matrix $A^{(i)}_1$ is scalar, this case can be regarded as $m_i=0$ by applying addition.

If $n^{(i)}_{1} \neq n$, then $k ^{(i)} \geq 2$ and we have $n= n^{(i)}_{1} +\dots + n^{(i)}_{k ^{(i)} } \geq k ^{(i)} n^{(i)}_{k ^{(i)} }$, $n^{(i)}_{k ^{(i)} ,1} \leq n^{(i)}_{k ^{(i)} }$ and
\begin{equation}
2- \frac{n^{(i)}_{1} +n^{(i)}_{1,1}}{n} = 2- \frac{n^{(i)}_{k ^{(i)}} +n^{(i)}_{k ^{(i)},1}}{n} \geq 2- \frac{2n^{(i)}_{k ^{(i)}}}{n} \geq 2 -\frac{2}{k ^{(i)}} .
\label{eq:mi1ineq}
\end{equation}
Hence we have $2- (n^{(i)}_{1} +n^{(i)}_{1,1})/{n} \geq 1$ and the equality holds iff $k ^{(i)} =2$, $p ^{(i)}_1 = p ^{(i)} _2 =1$ and $n^{(i)}_{1} =n^{(i)}_2 =n/2$.

We consider the case $k ^{(i)}=1$ for all $i$.
Then $\sum _{i=0}^r (1-1/p^{(i)}_{1}) =2$ and we have solution for the cases $r=3$ and $r=2$.
If $r=3$, then $p^{(0)}_{1} = p^{(1)}_{1}=p^{(2)}_{1}=p^{(3)}_{1}=2$, i.e. the case $\{ (d ,d)$, $(d,d)$, $(d,d)$, $(d,d) \}$ $(d=n/2)$.
If $r=2$, then $(p^{(0)}_{1}, p^{(1)}_{1}, p^{(2)}_{1})= (3,3,3) , (2,4,4), (2,3,6)$ or their permutations, i.e. the cases $\{ (d ,d, d), (d,d,d) , (d,d,d) \}$ $(d=n/3)$, $\{ (2d ,2d), (d,d,d,d) , (d,d,d,d) \}$ $(d=n/4)$ or $ \{ (3d ,3d), (2d,2d,2d) , (d,d,d,d,d,d) \}$ $(d=n/6)$.

We consider the case $\# \{ i \: | \: k ^{(i)} \geq 2\} \geq 2$.
Then it follows from Eq.(\ref{eq:mi1ineq}) that $\# \{ i \: | \: k ^{(i)} \geq 2 \} = 2 $, $k ^{(0)} = k ^{(1)} = 2$, $p ^{(i)}_1 = p ^{(i)} _2 =1$ and $n^{(i)}_{1} =n^{(i)}_2 =n/2$ $(i=0,1) $.
Hence we obtain the case $\{ (d ,d)- ((d),(d)) ,  \;  (d ,d)- ((d),(d)) \}$ $(d=n/2)$.

We consider the case $\# \{ i \: | \: k ^{(i)} \geq 2 \} =1$ and $r\geq 2$.
We set $k ^{(0)} \geq 2 $, $k ^{(1)} =\dots = k ^{(r)} =1$ for simplicity.
It follows from Eq.(\ref{eq:2-nn2}), $2- (n^{(0)}_{1} +n^{(0)}_{1,1})/{n} \geq 1 $ and $2- (n +n^{(i)}_{1,1})/{n} \geq 1/2 $ $(i\geq 1)$ that $r=2$, $n^{(0)}_{1} =n^{(0)}_2 =n/2$, $p^{(1)}_{1} =p^{(2)}_{1} = 2$ and $n^{(i)}_{1,1}= n^{(i)}_{1,2}
=n/2$ $(i=1,2)$. 
It corresponds to the case $\{ (d ,d)- ((d),(d)) ,  \;  (d ,d), \;  (d,d) \}$ $(d=n/2)$.

The remaining case is $\# \{ i \: | \: k ^{(i)} \geq 2 \} =1$ and $r= 1$.
We set  $k ^{(0)} \geq 2 $, $k ^{(1)} =1$ for simplicity.
By Eq.(\ref{eq:2-nn2}), we have 
\begin{equation}
\frac{n^{(0)}_{1} +n^{(0)}_{1,1}}{n} =1- \frac{1}{p^{(1)}_{1}}.
\end{equation}
It follows from $p^{(1)}_{1} \geq 2$ and Eq.(\ref{eq:mi1ineq}) that $2- (n^{(0)}_{1} +n^{(0)}_{1,1})/{n} \leq 3/2$ and $k^{(0)} \leq 4$.

If $k^{(0)} = 4$, then inequalities in Eq.(\ref{eq:mi1ineq}) must be equalities, $p ^{(0)}_1 = p ^{(0)} _2 =p ^{(0)}_3 = p ^{(0)} _4 =1$ and $n^{(0)}_{1} =n^{(0)}_2 =n^{(0)}_{3} =n^{(0)}_4 =n/4 $.
We have $p^{(1)}_{1} =2$ and $n^{(1)}_{1} =n^{(1)}_2 =n/2$.
Hence we obtain the case $\{ (d ,d,d,d)- ((d),(d),(d),(d)) ,  \;  (2d ,2d) \}$ $(d=n/4)$.

If $k^{(0)} = 3$, then we have $(n^{(0)}_{1} +n^{(0)}_{1,1})/{n} \leq 2/3 $.
Since $(n^{(0)}_{1} +n^{(0)}_{1,1})/{n} + 1/p^{(1)}_{1} =1$, we have $p^{(1)}_{1} =3 $ or $2$, i.e. $(n^{(0)}_{1} +n^{(0)}_{1,1})/{n} = 2/3$ or $1/2$.
If $(n^{(0)}_{1} +n^{(0)}_{1,1})/{n} = 2/3$, then $p^{(0)}_{1} =  p ^{(0)} _2 =p ^{(0)}_3 = 1$, $n^{(0)}_{i} =n^{(0)}_{i,1}=n/3$ $(i=1,2,3)$, $p^{(1)}_{1} =3$ and $n^{(1)}_{1} =n^{(1)}_2 =n^{(1)}_3 =n/3$, i.e. the case $\{ (d ,d,d)- ((d),(d),(d)) , \;  (d,d,d) \}$ $(d=n/3)$.
If $(n^{(0)}_{1} +n^{(0)}_{1,1})/{n} = 1/2$, then we have
\begin{align}
& 2n^{(0)}_{1,1} (p^{(0)}_1 +1)= 2n^{(0)}_{2,1} (p^{(0)}_2 +1)= 2n^{(0)}_{3,1} (p^{(0)}_3 +1)= n \\
& = n^{(0)}_{1} +n^{(0)}_{2} +n^{(0)}_{3}= n^{(0)}_{1,1}p^{(0)}_1 +n^{(0)}_{2,1}p^{(0)}_2 +n^{(0)}_{3,1}p^{(0)}_3 . \nonumber
\end{align}
By summing up, we have $2( n^{(0)}_{1,1} +n^{(0)}_{2,1} +n^{(0)}_{3,1} )=n^{(0)}_{1,1}p^{(0)}_1 +n^{(0)}_{2,1}p^{(0)}_2 +n^{(0)}_{3,1}p^{(0)}_3 $.
Since $n^{(0)}_{1} \geq n^{(0)}_{2} \geq n^{(0)}_{3}$, we have $n^{(0)}_{1,1} \leq n^{(0)}_{2,1} \leq n^{(0)}_{3,1}$, $p^{(0)}_1 \geq p^{(0)}_2 \geq p^{(0)}_3$ and $p^{(0)}_3 \leq 2$.
If $p^{(0)}_3=2$, then $p^{(0)}_1=p^{(0)}_2=2$.
If $p^{(0)}_3=1$, then we have $2n^{(0)}_{3,1}=n^{(0)}_{1,1}(p^{(0)}_1+1)=n^{(0)}_{2,1}(p^{(0)}_2+1)$ and $3 n^{(0)}_{3,1} = n^{(0)}_{1,1}p^{(0)}_1 +n^{(0)}_{2,1}p^{(0)}_2$.
Hence $n^{(0)}_{3,1} =n^{(0)}_{1,1} +n^{(0)}_{2,1}$
 and $1/(p^{(0)}_1 +1) +1/(p^{(0)}_2 +1) =1/2$.
Therefore $(p^{(0)}_1+1, p^{(0)}_2 +1, p^{(0)}_3+1)=(3,3,3) $, $(6,3,2)$ or $(4,4,2)$.
If $(p^{(0)}_1+1, p^{(0)}_2 +1, p^{(0)}_3+1)=(3,3,3) $, then $n^{(0)}_{1,1} =n^{(0)}_{2,1} =n^{(0)}_{3,1} =n/3$ and we obtain the case $\{ (2d,2d,2d) -((d,d), (d,d), (d,d)), \; (3d,3d) \}$ $(d=n/6 )$.
If $(p^{(0)}_1+1, p^{(0)}_2 +1, p^{(0)}_3+1)=(6,3,2) $, $n^{(0)}_{1,1} =n/12$, $n^{(0)}_{2,1} = n/6$ and $n^{(0)}_{3,1} =n/4$ and we obtain the case $\{ (5d,4d,3d) -((d,d,d,d,d), (2d,2d), (3d)), \; (6d,6d) \}$ $(d=n/12)$.
If $(p^{(0)}_1+1, p^{(0)}_2 +1, p^{(0)}_3+1)=(4,4,2) $, then $n^{(0)}_{1,1} =n^{(0)}_{2,1} =n/8$, $n^{(0)}_{3,1} =n/4$, and we obtain the case $\{ (3d,3d,2d) -((d,d,d), (d,d,d), (2d)), \; (4d,4d) \}$ $(d=n/8)$.

We investigate the case $k^{(0)} = 2$.
It follows from Eq.(\ref{eq:2-nn2}) and $n= n^{(0)}_{1,1}p^{(0)}_1 +n^{(0)}_{2,1}p^{(0)}_2 $ that
\begin{align}
& n^{(0)}_{1,1} (p^{(0)}_1 +1)= n^{(0)}_{2,1} (p^{(0)}_2 +1)= \left( 1-\frac{1}{p^{(1)}_1} \right) (n^{(0)}_{1,1}p^{(0)}_1 +n^{(0)}_{2,1}p^{(0)}_2) .
\end{align}
By erasing the term $p^{(0)}_1$, we obtain  
\begin{equation}
n^{(0)}_{2,1} +(p^{(1)}_1 -1) n^{(0)}_{1,1} = (p^{(1)}_1 -2) n^{(0)}_{2,1} p^{(0)}_2.
\end{equation}
Hence $p^{(1)}_1 \geq 3$.
It follows from $n^{(0)}_{1} \geq n^{(0)}_{2}$ that $n^{(0)}_{1,1} \leq n^{(0)}_{2,1}$, $p^{(0)}_1 \geq p^{(0)}_2$, $p^{(1)}_1 \geq 1+(p^{(1)}_1 -1)n^{(0)}_{1,1}/n^{(0)}_{2,1} = (p^{(1)}_1 -2)p^{(0)}_2$ and $p^{(0)}_2 \leq p^{(1)}_1 /(p^{(1)}_1 -2) $.
We consider the case $p^{(1)}_1 = 3$.
Then $p^{(0)}_2 \leq 3 $.
If $p^{(0)}_2=1$, then $n^{(0)}_{1,1}=0$ and it cannot occur.
If $p^{(0)}_2=2$, then $n^{(0)}_{2,1}=2n^{(0)}_{1,1}$ and $p^{(0)}_1=5$.
It corresponds to the case $\{ (5d,4d) -((d,d,d,d,d), (2d,2d)), \; (3d,3d,3d) \}$ $(d=n/9)$.
If $p^{(0)}_2=3$, then $n^{(0)}_{1,1}=n^{(0)}_{2,1}$ and $p^{(0)}_1=3$.
It corresponds to the case $\{ (3d,3d) -((d,d,d), (d,d,d)), \; (2d,2d,2d) \}$ $(d=n/6 )$.
We consider the case $p^{(1)}_1 \geq 4$.
Then $p^{(0)}_2 \leq p^{(1)}_1 /(p^{(1)}_1 -2) \leq 2$.
If $p^{(0)}_2=2$, then $p^{(1)}_1 =4$, $n^{(0)}_{1,1}=n^{(0)}_{2,1}$ and $p^{(0)}_1=2$.
It corresponds to the case $\{ (2d,2d) -((d,d), (d,d)), \; (d,d,d,d) \}$ $(d=n/4)$.
If $p^{(0)}_2 =1$, then $n^{(0)}_{1,1}=n^{(0)}_{2,1} (p^{(1)}_1 -3) /(p^{(1)}_1 -1) $ and $p^{(0)}_1 = 1+4/(p^{(1)}_1 -3) $.
Hence $4$ is divisible by $p^{(1)}_1 -3 $ and we have $p^{(1)}_1 =4,5$ or $7$.
If $p^{(1)}_1  =4$, then $p^{(0)}_1=5$ and  $3n^{(0)}_{1,1}=n^{(0)}_{2,1}$.
It corresponds to the case $\{ (5d,3d) -( (d,d,d,d,d),(3d)), \; (2d,2d,2d,2d) \} $ $(d=n/8)$.
If $p^{(1)}_1  =5$, then $p^{(0)}_1=3$ and  $2n^{(0)}_{1,1}=n^{(0)}_{2,1}$.
It corresponds to the case $\{ (3d,2d) -((d,d,d),(2d)), \; (d,d,d,d,d) \} $ $(d=n/5)$.
If $p^{(1)}_1  =7$, then $p^{(0)}_1=2$ and  $3n^{(0)}_{1,1}=2n^{(0)}_{2,1}$.
It corresponds to the case $\{ (4d,3d) -((2d,2d),(3d)), \; (d,d,d,d,d,d,d) \} $ $(d=n/7)$.

Thus we have exhausted all the cases.
\end{proof}
Remark that some patterns in the list of Proposition \ref{prop:classif} may not be realized as an irreducible system of differential equations.
In fact the patterns corresponding to Fuchsian systems as
\begin{align}
 & \quad \{ (d,d) , \; (d, d), \; (d,d) , \; (d,d) \} , \\
 & \quad \{ (d,d,d) , \; (d, d, d), \; (d,d,d) \},  \nonumber \\
 & \quad \{ (2d,2d) , \; (d,d,d,d) , \; (d, d, d ,d) \} ,  \nonumber \\
 & \quad \{ (3d,3d), \; (2d,2d,2d) , \; (d,d, d,d, d,d)\} , \nonumber 
\end{align}
for the case $d\geq 2$ (resp. $d=1$) cannot be realized (resp. can be realized) as irreducible systems, which was established by Kostov \cite{Kos} and Crawley-Boevey \cite{CB}.
\section{Concluding remarks} \label{sec:future}

We give comments for future reference.

In this paper we gave a tentative definition of the index of rigidity for systems of linear differential equations which may have irregular singularities, and we should clarify the correctness (or incorrectness) of our definition.
A key point would be Conjecture \ref{con:indrig}, which is compatible with Proposition \ref{prop:indrig}, and we should consider the case that the coefficient matrices are not semi-simple.

Crawley-Boevey \cite{CB} made a correspondence between systems of Fuchsian differential equations and roots of Kac-Moody root systems, which was applied for solving additive Deligne-Simpson problem.
Boalch \cite{Boa} gave a generalization of Crawley-Boevey's work to the cases which include an irregular singularity.
It would be hopeful to develop studies on this direction to understand several properties of differential equations.

Laplace transformation (or Fourier transformation) has been a powerful tool for analysis of differential equations.
It is known that Okubo normal form fits well with Laplace transformation.
In fact, Okubo normal form 
\begin{equation}
(xI_n - T) \frac{d\Psi }{dx} =A \Psi ,
\label{eq:okubo1}
\end{equation}
is transformed to
\begin{equation}
\frac{dV }{dz} =\left( T - \frac{A+I_n}{z} \right) V ,
\label{eq:Birkhoff}
\end{equation}
which is Birkhoff canonical form of Poincar\'e rank one by Laplace transformation 
\begin{equation}
V(z)= \int_C \exp (z x) \Psi (x) dx .
\end{equation}
Assume that the matrices $T$ and $A$ are written as
\begin{align}
& 
T= 
\left.
\left( 
\begin{array}{cccc}
t_1 I_{n_1} & 0 & \dots & 0  \\
0 & t_2 I_{n_2} & \dots & 0  \\
0 & 0 & \ddots & \vdots \\
0 & \cdots  & 0 & t_{k} I_{n_{k}} 
\end{array}
\right)
\right. , \;
 A  = 
\left.
\left( 
\begin{array}{cccc}
A^{[1,1]} & A^{[1,2]} & \dots & A^{[1,k]}   \\
A^{[2,1]} & A^{[2,2]} & \dots & A^{[2,k]}   \\
\vdots & \vdots & \ddots & \vdots \\
A^{[k,1]} & A^{[k,2]}& \dots & A^{[k,k]} 
\end{array}
\right)
\right. ,
\label{eq:Tblk}
\end{align}
where $t_i \neq t_j \; (i\neq j)$ and $A^{[i,j]}$ is a $n_i \times n_j$ matrix, and we further assume that $A$, $ A^{[i,i]}$ $(i=1, \dots ,k)$ are semi-simple and Eq.(\ref{eq:okubo1}) is irreducible.
Then the index of rigidity for Okubo normal form (Eq.(\ref{eq:okubo1})) is equal to
\begin{align}
\sum _{j=1}^k \left( (n_{j})^2 + \dim (Z(A^{[j,j]})) \right) + \dim (Z(A))  - n^2 
, \label{eq:idxOB}
\end{align}
which was described by Haraoka \cite{Har} and Yokoyama \cite{Yok}.
On the other hand, the index of rigidity of Eq.(\ref{eq:Birkhoff}) can be calculated as a special case of section \ref{sec:mileq1} (see Eq.(\ref{eq:idxrig})), and it is equal to Eq.(\ref{eq:idxOB}).
Hence the index of rigidity in this paper fits well with Laplace transformation.

We now observe an example of Laplace transformation.
Set
\begin{equation}
T = \left(
\begin{array}{cc}
0 & 1 \\
0 & 0  
\end{array}
\right) , \quad A = - \left(
\begin{array}{cc}
a_{1,1} +1 & a_{1,2} \\
a_{2,1} & a_{2,2} +1  
\end{array}
\right) ,
\label{eq:Anilp0}
\end{equation}
in Eq.(\ref{eq:Birkhoff}), which corresponds to Eq.(\ref{eq:Anilp}).
It follows from irreducibility that $a_{2,1} \neq 0$.
By (inverse) Laplace transformation, we obtain Eq.(\ref{eq:okubo1}), which is rewritten as
\begin{equation}
 \frac{d\Psi }{dx} =\left\{
- \frac{1}{x^2}
\left(
\begin{array}{cc}
a_{2,1} & a_{2,2} +1 \\
0 & 0
\end{array}
\right)
- \frac{1}{x}
\left(
\begin{array}{cc}
a_{1,1} +1 & a_{1,2} \\
a_{2,1} & a_{2,2} +1  
\end{array}
\right)
 \right\} \Psi ,
\label{eq:okubo2}
\end{equation}
and it can be reduced to a scalar differential equation by middle convolution as the example in section \ref{subsec:ex}.
Therefore we should develop a theory of Laplace transformation as well as the theory of middle convolution which is based on Euler's integral transformation (Theorem \ref{thm:DRintegrepr}).

Several important functions are written as a solution of single differential equations of higher order
\begin{align}
y^{(n) } + a_1 (z) y^{(n-1)} + \dots + a_{n-1} (z) y' + a_n(z) y=0 , \label{eq:sglDE}
\end{align}
where $a_ i(z)$ $(i=1,\dots ,n)$ are rational functions which may have poles at prescribed points $\{ t_1, \dots ,t_r \}$. 
Note that we need to treat delicately on writing Eq.(\ref{eq:sglDE}) into the form of systems of differential equations (\ref{eq:DEirr0}) to reflect the depth of the singularities.
Oshima \cite{O2} formulated middle convolution (Euler's transformation) for single differential equations of higher order, and Hiroe \cite{Hir} studied it for the case that the differential equation has an irregular singularity at $z=\infty $
and regular singularities.
Hiroe's result includes a part of the content of section \ref{sec:mileq1} in this paper with a different situation.
Moreover he clarified a structure of Kac-Moody root system, which is based on Boalch's study \cite{Boa}.
Studies on this direction should be developed further.

\section*{Acknowledgments}
The author would like to thank Hiroshi Kawakami and Daisuke Yamakawa for discussions and comments.
He was supported by the Grant-in-Aid for Young Scientists (B) (No. 19740089) from the Japan Society for the Promotion of Science.

\end{document}